\documentclass{amsart}

\usepackage{amsmath,amssymb,amsfonts,amsthm}
\usepackage{mathrsfs}
\usepackage{hyperref}
  \hypersetup{colorlinks=true,urlcolor=blue,linkcolor=red,citecolor=red}
\usepackage[capitalise,noabbrev]{cleveref}
\usepackage{enumitem}
  \setlist{nosep}
  \setlist[enumerate]{label={(\roman*)},
  align=left,leftmargin=1.66\parindent,labelwidth=1.66\parindent,labelsep=0pt}
  \setlist[itemize]{label={\raisebox{0.15ex}{\tiny \textbullet}},
  align=left,leftmargin=\parindent,labelindent=0.33\parindent,labelwidth=0.67\parindent,
  labelsep=0pt,listparindent=0pt}
\usepackage{tikz}
  \usetikzlibrary{arrows}
\usepackage{pgfplots}
  \pgfplotsset{compat=1.13}

\theoremstyle{plain}
\newtheorem{theorem}{Theorem}[section]
\newtheorem{corollary}[theorem]{Corollary}
\newtheorem{lemma}[theorem]{Lemma}
\newtheorem{proposition}[theorem]{Proposition}

\theoremstyle{definition}
\newtheorem{assumption}{Assumption}
\newtheorem{definition}[theorem]{Definition}

\theoremstyle{remark}
\newtheorem{remark}[theorem]{Remark}
\newtheorem{example}[theorem]{Example}

\renewcommand{\geq}{\geqslant}
\renewcommand{\leq}{\leqslant}

\newcommand{\R}{\mathbb{R}}

\newcommand{\Z}{\mathbb{Z}}
\newcommand{\Q}{\mathbb{Q}}
\newcommand{\torus}[1]{\mathbb{R}^{#1} / \mathbb{Z}^{#1}}
\newcommand{\Acal}{\mathscr{A}}
\newcommand{\Bcal}{\mathscr{B}}
\newcommand{\NASmin}{\mathfrak{A}}
\newcommand{\NASmax}{\mathfrak{B}}
\newcommand{\X}[2][]{X^{#1}_{#2}}
\newcommand{\Y}[2][]{y^{#1}_{#2}}
\newcommand{\abs}[1]{\left| {\ifx\\#1\\ \cdot \else #1 \fi} \right|}
\newcommand{\norm}[1]{| {\ifx\\#1\\ \cdot \else #1 \fi} |}
\newcommand{\Norm}[1]{\left| {\ifx\\#1\\ \cdot \else #1 \fi} \right|}
\newcommand{\supnorm}[1]{\left\| \ifx\\#1\\ \cdot \else #1 \fi \right\|_\infty}

\newcommand{\Cper}[1]{\mathcal{C}^{0}_\text{per}(#1)}
\newcommand{\dist}[2]{\operatorname{dist}_{#1}(#2)}
\newcommand{\NP}[2]{{N \! P}_{#1}(#2)}
\newcommand{\Proj}[2]{{P}_{#1} #2}
\newcommand{\proj}[2]{{p}_{#1}(#2)}
\newcommand{\varb}{b'}
\newcommand{\argt}{\centerdot} 
\newcommand{\indic}[1]{\mathbf{1}_{#1}}
\newcommand{\clo}[1]{\overline{#1}} 
\newcommand{\scalar}[2]{\langle #1, #2 \rangle}
\newcommand{\bigp}[1]{\big( \, #1 \, \big)}

\DeclareMathOperator{\argmin}{arg\,min}
\DeclareMathOperator{\argmax}{arg\,max}

\title[Unique ergodicity of zero-sum differential games]{Unique ergodicity of \\
deterministic zero-sum differential games}
\author[A.\ Hochart]{Antoine Hochart}
\address{Facultad de Ingenier\'ia y Ciencia, Universidad Adolfo Ib\'a\~nez,
Diagonal Las Torres 2640, Santiago, Chile}
\email{antoine.hochart@gmail.com}
\thanks{The author is supported by FONDECYT grant 3180662.}

\date{January 7, 2020}



\begin{document}

\maketitle

\begin{abstract}
  We study the ergodicity of deterministic two-person zero-sum differential games.
  This property is defined by the uniform convergence to a constant of either
  the infinite-horizon discounted value as the discount factor
  tends to zero, or equivalently, the averaged finite-horizon value
  as the time goes to infinity.
  We provide necessary and sufficient conditions for the unique ergodicity of a game.
  This notion extends the classical one for dynamical systems, namely when
  ergodicity holds with any (suitable) perturbation of the running payoff function.
  Our main condition is symmetric between the two players and involve dominions, i.e.,
  subsets of states that one player can make approximately invariant.
\end{abstract}

\section{Introduction}

We study the ergodic problem for deterministic two-player zero-sum differential games.
Such games are defined by a nonlinear system in $\R^n$ controlled by two players,
\begin{equation}
  \label{eq:controlled-system}
  \begin{cases}
    \dot{\X{t}} = f(\X{t},a_t,b_t) , \quad t > 0 , \\
    \X{0} = x ,
  \end{cases}
\end{equation}
where the first player chooses the actions $a_t \in A$ and the second player,
the actions $b_t \in B$.
Given a continuous and bounded payoff function $\ell$, player~1 intends to minimize
one of the following payoff functionals, whereas player~2 intends to maximize it:
\begin{equation}
  \label{eq:discounted-functional}
  J_\delta(x,a,b) = \int_0^{\infty} e^{-\delta s} \ell(\X{s}, a_s, b_s) ds
\end{equation}
in the infinite-horizon discounted game, or
\begin{equation}
  \label{eq:finite-horizon-functional}
  J(t,x,a,b) = \int_0^t \ell(\X{s}, a_s, b_s) ds
\end{equation}
in the game played in finite horizon $t$.
We assume that the data are $\Z^n$-periodic in the state variable $x \in \R^n$ so that
the state space can be identify with the $n$-torus $\torus{n}$.
We also restrict our study to the lower game, in which player~1 adapts her control to
player~2's actions, but note that all the results can be readily adapted to the upper game or
the situation in which the classical Isaacs condition holds.

The value of the discounted (lower) game and the one of the finite-horizon (lower) game,
denoted respectively by $v_\delta(x)$ and $v(t,x)$, are the payoffs at equilibrium and can be
characterized as the viscosity solutions of, respectively, a stationary Hamilton-Jacobi PDE
and an evolutionary Hamilton-Jacobi PDE involving the Hamiltonian of the (lower) game.

\medskip
The ergodic problem for zero-sum differential games or for Hamilton-Jacobi equations,
its PDE counterpart, concerns the asymptotic behavior of the value functions
$v_\delta(x)$ and $v(t,x)$.
More precisely, it deals with the uniform convergence toward a constant of
$\delta v_\delta(x)$ when the discount factor $\delta$ goes to zero, and of
$v(t,x)/t$ when the horizon $t$ goes to infinity.
The problem has been much studied since the seminal work of Lions, Papanicolaou and
Varadhan \cite{LPV87}.
For optimal control (i.e., one-player) problems, let us mention the work of
Arisawa \cite{Ari97,Ari98} and for two-player games, the one of Alvarez and Bardi \cite{AB03,AB07}
or Cardaliaguet \cite{Car10}. 
More recently, the ergodic control problem has been studied by Quincampoix and Renault \cite{QR11},
Gaitsgory and Quincampoix \cite{GQ13}, Cannarsa and Quincampoix \cite{CQ15} or
Buckdahn, Quincampoix and Renault \cite{BQR15}, for situations in which the limit value is not
necessarily constant with respect to the initial state.
Let us further mention the work of Khlopin \cite{Khl17} on Abelian-Tauberian properties, or
Ziliotto \cite{Zil17,Zil19} on counterexamples to the convergence of the values, which illustrate
the connection between the discrete setting (i.e., repeated games) and the continuous setting
(which we consider here)\footnote{Note that the counterexample to Hamilton-Jacobi
homogenization given in \cite{Zil17} has been preceded by counterexamples for the convergence of
the value of repeated games given by Vigeral \cite{Vig13} and Ziliotto \cite{Zil16}.}.


An important problem is then to characterize the differential games which are ergodic.
Typical results require that the nonlinear system (or a subsystem, if it is decomposable)
be uniformly controllable by one player, that is, any point $x$ is controllable to any
other point $y$ by this player, either exactly or approximately, asymptotically or
in bounded time (see e.g., \cite{Ari98,Bet05,AB07}).
Such conditions are independent of the payoff function $\ell$ and thus imply that
the game is in fact uniquely ergodic.
The latter notion, which was originally defined for dynamical systems, readily extends to
differential games: a game is uniquely ergodic if it is ergodic for all perturbations of
the payoff function $\ell$ that only depend on the state variable.
In \cite{Ari97}, Arisawa showed that a converse property holds for systems controlled by
one player and proved the existence of an ergodic attractor when unique ergodicity holds.
But for two-player games, these controllability conditions totally lack symmetry,
focusing only on one player.

\medskip
The purpose of this article is to study the unique ergodicity property
for differential games.
We introduce a ``dominion condition'' which is in essence symmetrical
between the two players.
Each dominion is associated with one player and, roughly speaking, corresponds to a nonempty
subset of states that this player can make approximately invariant for the dynamics.
We show that if a game is uniquely ergodic, then the players do not have disjoint dominions.
To prove this result, we use an Hamilton-Jacobi PDE approach.
Under specific controllability assumptions (independence of $f$ with respect to
the state variable or uniform time estimates on the dynamics) we further prove that
the ``dominion condition'' is in fact equivalent to unique ergodicity.
Thus our results generalize the unique ergodicity property of dynamical systems\footnote{We refer
the reader to \cite[Sec.~6.1]{AB03} for the connections between classical ergodic theory and
ergodicity of games or Hamiltonians.}, as well as the analysis of Arisawa in \cite{Ari97,Ari98}
for optimal control problems.
In particular, let us observe that if a system is uniformly controllable by one player, then,
whatever assumptions are made on the controllability (asymptotic or bounded time,
exact or approximate), it implies that the other player has a unique trivial dominion,
namely the whole state space, and so that the ``dominion condition'' trivially holds.

We finally mention that the notion of dominion coincides with the ones of leadership domain and
discriminating domain in viability theory (see e.g., \cite{Car96}), and therefore
also relates with the notions of \textbf{B}-set and approachability in repeated games
with vector payoffs, as shown by As Soulaimani, Quincampoix and Sorin in \cite{ASQS09}. 
However, the ideas developed in this article were first inspired by the study of the ergodic
problem for zero-sum repeated games, i.e., games played in discrete time (see the companion
articles \cite{Hoc19} and \cite{AGH19}).
In order to remain consistent with the latter work, we have chosen to use the term ``dominion''
instead of ``domain'', although the two terms could be interchanged.

\medskip
The paper is organized as follows.
\Cref{sec:preliminaries} is dedicated to preliminaries on differential games, their value
functions and the Hamilton-Jacobi PDE approach to ergodicity.
This section only provides some notation and classical results.
It can therefore be safely skipped by readers familiar with the subject.
In \Cref{sec:ergodicity-Hamiltonians}, we introduce and study the unique ergodicity
property for general Hamiltonians, that is, the property of an Hamiltonian to be ergodic
for any suitable perturbation.
This (slightly) generalizes a characterization by Alvarez and Bardi in \cite{AB10}.
In \Cref{sec:PDE-approach}, we introduce the notion of dominion and study the unique
ergodicity property for differential games following a PDE approach.
In \Cref{sec:controllability-approach}, we study the unique ergodicity of
differential games relying only on a dynamical system approach.
Finally in \Cref{sec:dominion-characterization}, we characterize dominions
in operator-theoretic terms, which establishes the link with the notion of
discriminating / leadership domain in viability theory.

\section{Preliminaries}
\label{sec:preliminaries}

We introduce here some notation as well as standard definitions and results on differential games
and their PDE approach.
Readers familiar with the subject can safely skip the section.

\subsection{Framework and standing assumptions}
\label{sec:setting}

We start by describing the setting of deterministic two-player zero-sum
differential games that we study in this article.
Consider first the controlled nonlinear system~\labelcref{eq:controlled-system}
where the map $f$ is from $\R^n \times A \times B$ to $\R^n$, with $A, B$ nonempty
compact metric spaces.
We assume throughout the paper that $f$ is continuous in all variables and Lipschitz continuous
in the state variable, uniformly in the control variables, i.e., denoting by $\Norm{}$
the standard Euclidean norm,
\[
  \Norm{f(x,a,b)-f(y,a,b)} \leq L_f \Norm{x-y}
\]
for some constant $L_f \geq 0$ and for all $x, y \in \R^n$, $a \in A$ and $b \in B$.
Player~1 (resp., player~2) chooses a control $t \mapsto a_t$ (resp., $t \mapsto b_t$) in
the set of Lebesgue measurable functions from $[0,+\infty)$ to $A$ (resp., $B$), which
we denote by $\Acal$ (resp., $\Bcal$)\footnote{In order to simplify the notation,
we shall equally denote by $a$ and $b$ single elements of $A$ and $B$, respectively,
and controls of player~1 and player~2, i.e., elements of $\Acal$ and $\Bcal$, respectively.
The distinction should be clear from the context}.
The Cauchy-Lipschitz theorem implies that
\Cref{eq:controlled-system} has a unique solution, which we denote by $\X[x,a,b]{t}$ and
for which the differential equation holds for almost all $t > 0$.

We are further given a bounded continuous payoff function
$\ell: \R^n \times A \times B \to \R$, where we let $M_\ell = \supnorm{\ell}$,
the supremum norm of $\ell$.
Then, for any trajectory of the controlled system \labelcref{eq:controlled-system},
we mainly consider in this article the discounted payoff
functional~\labelcref{eq:discounted-functional}, associated with the game played in infinite
horizon with a discount factor $\delta > 0$ on the running payoff.
The objective of player~1 is to minimize the latter functional, whereas player~2 intends
to maximize it.
We shall also briefly mention the payoff functional~\labelcref{eq:finite-horizon-functional}
associated with the game played in a finite horizon $t > 0$.

Additionally to the classical conditions on $f$ and $\ell$ already mentioned above (and which
we reproduce below), we make throughout the paper the following assumption.
Before stating it, let us recall that a modulus of continuity is a nondecreasing function
$\omega : [0,+\infty) \to [0,+\infty)$, vanishing and continuous at $0$, that is,
such that $\lim_{r \to 0} \omega(r) = \omega(0) = 0$.

\begin{assumption}[Standing assumption]
  \label{asm:standing-assumption}
  \leavevmode
  \begin{enumerate}
    \item The function $f$ is continuous in all variables and uniformly Lipschitz
      continuous in the state variable, the function $\ell$ is bounded continuous,
      the action spaces $A$ and $B$ are nonempty compact sets.
    \item \label{itm:std-modulus}
      The payoff function $\ell$ is uniformly continuous with respect to the state variable,
      uniformly with respect to the control variables, i.e., there exists
      a modulus of continuity $\omega_\ell$ such that
      \[
        \abs{\ell(x,a,b)-\ell(y,a,b)} \leq \omega_\ell(\Norm{x-y}) ,
        \quad \forall x, y \in \R^n , \enspace \forall a \in A ,
        \enspace \forall b \in B .
      \]
    \item \label{itm:std-periodicity}
      The functions $f$ and $\ell$ are $\Z^n$-periodic in the state variable,
      i.e., for $\varphi \in \{f, \ell \}$,
      \[
        \varphi(x+k,a,b) = \varphi(x,a,b) ,
        \quad \forall k \in \Z^n , \enspace \forall x \in \R^n ,
        \enspace \forall a \in A , \enspace \forall b \in B .
      \]
  \end{enumerate}
\end{assumption}

Let us remark that \Cref{itm:std-periodicity} implies that the state space can be identify
with the $n$-torus $\torus{n}$.
Although we shall work mostly in $\R^n$, we draw the attention of the reader to the fact
that sometimes, we will consider objects in the quotient space.
Moreover, \Cref{itm:std-periodicity} together with the continuity of $f$ entails
the boundedness of this function.
We therefore let $M_f = \supnorm{f}$.

\subsection{Value functions and Hamilton-Jacobi PDEs}

We introduce here the concept of value function and then characterize it in terms of
viscosity solution of some Hamilton-Jacobi PDE.
We keep the presentation to a minimum and refer the reader to the classical monograph
\cite{BCD97} for more details.

Let us start with the definition of nonanticipating strategies.

\begin{definition}[Nonanticipating strategy]
  A nonanticipating strategy for the first player is a map $\alpha: \Bcal \to \Acal$
  such that for any time $t > 0$ and any controls $b^1, b^2 \in \Bcal$ of player~2,
  if $b^1_s = b^2_s$ for almost all $s \leq t$ then
  $\alpha[b^1]_s = \alpha[b^2]_s$ for almost all $s \leq t$.
  We denote by $\NASmin$ the set of nonanticipating strategies for player~1.

  The set $\NASmax$ of nonanticipating strategies $\beta: \Acal \to \Bcal$
  for the second player is defined accordingly.
\end{definition}

We then introduce the (unnormalized) value functions.
When player~2 chooses a control $b \in \Bcal$ and player~1 is allowed to adapt her response
to this control, i.e., when she chooses a nonanticipating strategy $\alpha \in \NASmin$,
we are considering the {\em lower game}, which we denote by $\Gamma^-$.
The lower value function associated with the infinite-horizon discounted payoff
functional is then defined by
\[
  v^-_\delta(x) = \inf_{\alpha \in \NASmin} \sup_{b \in \Bcal} J_\delta(x,\alpha[b],b) .
\]
On the other hand, if player~1 is bound to choose a control $a \in \Acal$ to which
player~2 can adapt by choosing a nonanticipating strategy $\beta \in \NASmax$,
then we are considering the {\em upper game}, denoted by $\Gamma^+$, and the upper value
function is given by
\[
  v^+_\delta(x) = \sup_{\beta \in \NASmax} \inf_{a \in \Acal} J_\delta(x,a,\beta[a]) .
\]
When the game is played in a finite horizon $t > 0$, the value functions are defined
similarly by, respectively,
\[
  v^-(t,x) = \inf_{\alpha \in \NASmin} \sup_{b \in \Bcal} J(t,x,\alpha[b],b)
  \quad \text{and} \quad
  v^+(t,x) = \sup_{\beta \in \NASmax} \inf_{a \in \Acal} J(t,x,a,\beta[a]) .
\]

We always have $v^-_\delta(x) \leq v^+_\delta(x)$ (resp., $v^-(t,x) \leq v^+(t,x)$) and
the differential game is said to have a value at state $x$ if there is equality.
The latter holds under the classical Isaacs condition (which we recall
at the end of the section).
However, in this work, we do not need to make such an assumption:
all the results presented in the article hold in the lower as well as in the upper game.
Owing to the symmetry of $\Gamma^-$ and $\Gamma^+$, we shall only consider from now on
the lower game, and therefore drop the ``-'' superscript for simplicity of the notation.
We leave to the reader the straightforward adaptation of the results to the upper game
(or to the situation in which Isaacs' condition holds).

\medskip
We readily deduce from the above definitions that the two normalized value functions
$x \mapsto \delta v_\delta(x)$ and $(t,x) \mapsto v(t,x) / t$ are bounded by $M_\ell$
and $\Z^n$-periodic.
It is also known that they are respectively continuous on $\R^n$ and
Lipschitz continuous on $[0,T] \times \R^n$ for all times $T > 0$.  
Furthermore, they can be characterized as viscosity solutions\footnote{In this paper,
the solutions of PDEs will always be in the continuous viscosity sense.} of some PDEs,
called Hamilton-Jacobi-Isaacs' equations.
These equations involve the (lower) {\em Hamiltonian}, defined by
\begin{equation}
  \label{eq:Hamiltonian}
  H(x,p) = H^{-}(x,p) = \min_{b \in B} \max_{a \in A}
  \big\{ -\scalar{f(x,a,b)}{p} - \ell(x,a,b) \big\} ,
  \quad x , p \in \R^n ,
\end{equation}
where $\scalar{\cdot}{\cdot}$ is the standard scalar product on $\R^n$.
The next result illustrates this fact.
Note that, given any real function $(t,x) \mapsto \varphi(t,x)$, we denote by
$\partial_t \varphi$ its partial derivative with respect to the time variable $t$,
and by $D \varphi$ its gradient with respect to the state variable $x$.

\begin{theorem}[see {\cite[Ch.~III, Prop.~2.8,~3.5]{BCD97}}]
  \label{thm:HJ-PDE-characterization}
  Under \Cref{asm:standing-assumption}, the value function $v_\delta$
  is the unique continuous viscosity solution of the Hamilton-Jacobi PDE
  \begin{equation}
    \label{eq:HJ-PDE-stationary}
    \tag{HJ$_\delta$}
    \begin{cases}
      \delta u(x) + H(x, Du(x)) = 0 , \quad \text{in} \enspace \R^n , \\
      u \enspace \Z^n \text{-periodic} ,
    \end{cases}
  \end{equation}
  and the value function $(t,x) \mapsto v(t,x)$ is the unique continuous viscosity solution of
  the Hamilton-Jacobi PDE
  \begin{equation}
    \label{eq:HJ-PDE-evolution}
    \tag{HJ$_\text{t}$}
    \begin{cases}
      \partial_t u(t,x) + H(x, Du(t,x)) = 0 , & \text{in} \enspace (0,+\infty) \times \R^n , \\
      u(0,x) = 0 , & \text{for all} \enspace x \in \R^n , \\
      u(t,\cdot) \enspace \Z^n \text{-periodic} , & \text{for all} \enspace t > 0 .
    \end{cases}
  \end{equation}
\end{theorem}

The upper value functions are characterized by the same PDEs after replacing the lower
Hamiltonian $H$ with the upper Hamiltonian
\[
  H^{+}(x,p) =  \max_{a \in A} \min_{b \in B} \big\{ -\scalar{f(x,a,b)}{p} - \ell(x,a,b) \big\} ,
  \quad x , p \in \R^n .
\]
Consequently, if {\em Isaacs' condition} holds, that is, if
\begin{multline}
  \label{eq:Isaacs-condition}
  \max_{a \in A} \min_{b \in B} \big\{ -\scalar{f(x,a,b)}{p} - \ell(x,a,b) \big\} \\
  = \min_{b \in B} \max_{a \in A} \big\{ -\scalar{f(x,a,b)}{p} - \ell(x,a,b) \big\},
  \quad \forall x , p \in \R^n ,
\end{multline}
then the lower and the upper value functions are equal.

\subsection{Ergodicity and PDE approach}

In this article, we are interested in the asymptotic behavior of the value functions,
that is, in the behavior of $v_\delta(x)$ as the discount factor $\delta$ goes to $0$
(resp., in the behavior of $v(t,x)$ as the time horizon $t$ goes to $+\infty$).
More specifically, we study the so-called {\em ergodic problem}, that is, the situation
in which there exists a constant $\lambda \in \R$ such that the normalized value
$\delta v_\delta(x)$ tends to $\lambda$ as $\delta$ goes to $0$ (resp., $v(t,x) / t$
tends to $\lambda$ as $t$ goes to $+\infty$) uniformly in $x$.
This property is called {\em ergodicity} of the game. 

Thanks to \Cref{thm:HJ-PDE-characterization}, the latter problem can be studied by a PDE approach.
With this in mind, we shall sometimes consider arbitrary Hamiltonians $(x,p) \mapsto H(x,p)$
defined on $\R^n \times \R^n$ that satisfy the following properties.
Note that these properties are inherited from the Hamiltonian defined
in \labelcref{eq:Hamiltonian}.

\begin{assumption}
  \label{asm:Hamiltonian}
  \leavevmode
  \begin{enumerate}
    \item \label{itm:Hamiltonian-continuity}
      The Hamiltonian $H : \R^n \times \R^n \to \R$ is continuous.
    \item \label{itm:Hamiltonian-periodicity}
      $H$ is $\Z^n$-periodic in the first variable, i.e,
      for all $x, p \in \R^n$ and $k \in \Z^n$,
      \[
        H(x+k,p) = H(x,p) .
      \]
    \item \label{itm:Hamiltonian-modulus}
      There is a modulus of continuity $\omega : [0,+\infty) \to [0,+\infty)$ such that,
        for all $x, y, p \in \R^n$,
      \[
        \abs{H(x,p) - H(y,p)} \leq \omega \big( \Norm{x-y} (1 + \Norm{p}) \big) .
      \]
    \item \label{itm:Hamiltonian-recession}
      There is a function $H_\infty: \R^n \times \R^n \to \R$ that is
      positively homogeneous of degree one in the second variable, and
      a constant $M_H \geq 0$ such that, for all $x, p \in \R^n$,
      \[
        \abs{H(x,p) - H_\infty(x,p)} \leq M_H .
      \]
  \end{enumerate}
\end{assumption}

Let us make few comments about these assumptions.
First, \Cref{itm:Hamiltonian-continuity,itm:Hamiltonian-periodicity,itm:Hamiltonian-modulus}
imply that the PDEs \labelcref{eq:HJ-PDE-stationary,eq:HJ-PDE-evolution} have a unique
continuous viscosity solution.
In particular, \Cref{itm:Hamiltonian-modulus} implies that
the comparison principle for viscosity solutions holds.
Second, the map $H_\infty$ introduced in \Cref{itm:Hamiltonian-recession} is called
the {\em recession function} of $H$.
The positive homogeneity of degree one means that
\[
  H_\infty (x, \nu p) = \nu H_\infty (x,p)
\]
for all $x, p \in \R^n$ and all $\nu > 0$.
A consequence is that
\[
  \lim_{\nu \to +\infty} \frac{H(x, \nu p)}{\nu} = H_\infty(x,p)
\]
uniformly in $(x,p)$, and so $H_\infty$ is necessarily unique,
continuous and $\Z^n$-periodic in the first variable.
Let us observe that if $H$ is the Hamiltonian associated with the lower game $\Gamma^-$,
as defined in \labelcref{eq:Hamiltonian}, then
\begin{equation}
  \label{eq:homogeneous-Hamiltonian}
  H_\infty(x,p) = \min_{b \in B} \max_{a \in A} \big\{ -\scalar{f(x,a,b)}{p} \big\} ,
  \quad x , p \in \R^n .
\end{equation}

Following a PDE approach, the existence and the value of the ergodic constant $\lambda$
can be related with the viscosity solutions of the following cell problem:
\begin{equation}
  \label{eq:cell-problem}
  \tag{CP}
  \begin{cases}
    c + H(x, Dw(x)) = 0 , \quad \text{in} \enspace \R^n , \\
    w \enspace \Z^n\text{-periodic} .
  \end{cases}
\end{equation}
The next result explains this connection.
In its statement, we abbreviate upper semicontinuous as u.s.c.\ and
lower semicontinuous as l.s.c.
Note that the result was shown in \cite{AB03} for second-order Hamilton-Jacobi PDEs.

\begin{theorem}[{\cite[Thm.~4]{AB03}}]
  \label{thm:ergodicity}
  Let $H$ be an arbitrary Hamiltonian satisfying
  \Cref{itm:Hamiltonian-continuity,itm:Hamiltonian-periodicity,itm:Hamiltonian-modulus}
  of \Cref{asm:Hamiltonian}.
  The following assertions are equivalent.
  \begin{enumerate}
    \item \label{itm:abstract-ergodicity-i}
      If $u_\delta$ is the solution of the stationary problem \labelcref{eq:HJ-PDE-stationary},
      then $\delta u_\delta(x)$ converges uniformly in $x$ to a constant $\lambda_1 \in \R$
      as $\delta$ goes to $0$.
    \item \label{itm:abstract-ergodicity-ii}
      If $u$ is the solution of the Cauchy problem \labelcref{eq:HJ-PDE-evolution}, then
      $u(t,x) / t$ converges uniformly in $x$ to a constant $\lambda_2 \in \R$
      as $t$ goes to $+\infty$.
    \item \label{itm:abstract-ergodicity-iii}
      There exists a constant $\lambda_3$ such that
      \begin{multline*}
        \sup \left\{ c \in \R \mid \text{there is an u.s.c.\ subsolution of
          \labelcref{eq:cell-problem}} \right\} \\
        = \lambda_3 =
        \inf \left\{ c \in \R \mid \text{there is a l.s.c.\ supersolution of
          \labelcref{eq:cell-problem}} \right\} .
      \end{multline*}
  \end{enumerate}
  Moreover, if one of these assertions is true, then $\lambda_1 = \lambda_2 = \lambda_3$.
\end{theorem}

When an arbitrary Hamiltonian $H$ satisfies one (hence all) of the above assertions,
we say that it is {\em ergodic}. 
We refer the reader to \cite[Sec.~6]{AB03} for a detailed discussion on the connections
between classical ergodic theory of deterministic dynamical systems and ergodicity of
Hamiltonians.

\section{Unique ergodicity of Hamiltonians}
\label{sec:ergodicity-Hamiltonians}

In this section, we introduce the central concept of this article, namely
{\em unique ergodicity}, which we first apply to arbitrary Hamiltonians.

Unique ergodicity is a property that originally applies to dynamical systems.
Although its definition (existence of a unique invariant probability measure)
cannot be readily extended to differential games or arbitrary Hamiltonians,
its characterization in terms of long time averages of any continuous function along
the trajectories makes this extension possible.

Alvarez and Bardi in \cite{AB10} used this terminology of unique ergodicity and
studied the property for two-player controlled systems.
However, we mention that before this work, the property was already studied for
controlled systems, without being given any explicit name
(see for instance \cite{Ari97,Ari98}). 

\subsection{Definition and characterization}

\begin{definition}[Uniquely ergodic Hamiltonian]
  Let $H: \R^n \times \R^n \to \R$ be an Hamiltonian satisfying
  \Cref{itm:Hamiltonian-continuity,itm:Hamiltonian-periodicity,itm:Hamiltonian-modulus}
  in \Cref{asm:Hamiltonian}.
  We say that $H$ is {\em uniquely ergodic} if, for every continuous and
  $\Z^n$-periodic function $g: \R^n \to \R$, the perturbed Hamiltonian $g+H$ is ergodic,
  i.e., one (hence all) of the assertions in \Cref{thm:ergodicity} holds with $g+H$.
\end{definition}

In the remainder, we denote by $\Cper{\R^n}$ the space of continuous and
$\Z^n$-periodic real functions over $\R^n$.

We next give a characterization of unique ergodicity which is
very similar to Proposition~$2.3$ in \cite{AB10} (as a matter of fact,
most of the proof is borrowed from the latter reference, which we have chosen
to reproduce for the sake of completeness).
However, our result differs from the one of Alvarez and Bardi in two ways.
First, it is not restricted to Hamiltonians associated with differential games
but it applies to arbitrary Hamiltonians.
Second, our definition of unique ergodicity is slightly more general, in the sense that
we only need to consider perturbations of Hamiltonians of the form $g \in \Cper{\R^n}$.

\begin{theorem}[compare with {\cite[Prop.~2.3]{AB10}}]
  \label{thm:unique-ergodicity-Hamiltonians}
  Let $H: \R^n \times \R^n \to \R$ be an Hamiltonian satisfying \Cref{asm:Hamiltonian}.
  It is uniquely ergodic if and only if the following assertions hold:
  \begin{itemize}
    \item {\em (Structural equicontinuity)} for every continuous and $\Z^n$-periodic
      function $g: \R^n \to \R$, if $u_\delta$ denotes the solution of
      \labelcref{eq:HJ-PDE-stationary} with the Hamiltonian $g+H$, then the family
      $\{ \delta u_\delta \}_{0 < \delta \leq 1}$ is equicontinuous;
    \item {\em (Strong maximum principle)} the constant functions are the only continuous
      viscosity solutions of the PDE
      \begin{equation}
        \label{eq:HJ-PDE-homogeneous}
        \tag{HJ$_\infty$}
        \begin{cases}
          H_\infty(x, Dw(x)) = 0 , \quad \text{in} \enspace \R^n , \\
          w \enspace \Z^n \text{-periodic} ,
        \end{cases}
      \end{equation}
      where $H_\infty$ is the recession function of $H$.
  \end{itemize}
\end{theorem}

\begin{proof}
  Let us first assume that $H$ is uniquely ergodic.
  Let $g \in \Cper{\R^n}$ and, for $\delta \in (0,1]$, let $u_\delta$ be the solution of
  \labelcref{eq:HJ-PDE-stationary} with the Hamiltonian $g+H$.
  Since $g+H$ satisfies \Cref{asm:Hamiltonian}, the standard comparison principle for
  viscosity solutions holds.
  A first straightforward application of this principle yields that
  the family $\{ \delta u_\delta \}_{0 < \delta \leq 1}$ is uniformly bounded by
  $M_g = \supnorm{g(\cdot) + H(\cdot,0)}$.
  Then, using this fact and once again the comparison principle, we get that
  \[
    \supnorm{u_\delta - u_{\delta'}} \leq M_g \abs{\frac{1}{\delta} - \frac{1}{\delta'}}
  \]
  for all $\delta, \delta' \in (0,1]$.
  Since the solution of \labelcref{eq:HJ-PDE-stationary} is continuous, we further deduce that
  the function $(\delta,x) \mapsto \delta u_\delta(x)$ is continuous on $(0,1] \times \R^n$.
  Together with the hypothesis that $\delta u_\delta(x)$ converges uniformly in $x$ to
  a constant when $\delta$ goes to 0, it entails the equicontinuity of
  $\{ \delta u_\delta \}_{0 < \delta \leq 1}$.

  To show that the second point (strong maximum principle) holds, let us consider any
  continuous viscosity solution $w$ of \labelcref{eq:HJ-PDE-homogeneous}.
  Fix $\rho > 0$ and denote by $u^\rho_\delta$ the solution of \labelcref{eq:HJ-PDE-stationary}
  with the Hamiltonian $-\rho w + H$, i.e., the solution of
  \begin{equation}
    \label{eq:HJ-PDE-perturbed}
    \begin{cases}
      \delta u(x) - \rho w(x) + H(x, Du(x)) = 0 , \quad \text{in} \enspace \R^n , \\
      u \enspace \Z^n \text{-periodic} .
    \end{cases}
  \end{equation}
  Let us show that $w^\rho_\delta = \frac{1}{\delta} (\rho w - M_H)$ (where $M_H$ is
  the constant defined in \Cref{itm:Hamiltonian-recession} of \Cref{asm:Hamiltonian}
  for the Hamiltonian $H$) is a viscosity subsolution of \labelcref{eq:HJ-PDE-perturbed}.
  To that end, for any $x \in \R^n$, let us consider any continuously differentiable
  function $\varphi$ such that $w^\rho_\delta - \varphi$ has a local maximum point at $x$.
  Then the function $w - \frac{\delta}{\rho} \varphi$ has also a local maximum at $x$,
  which implies that $H_\infty(x, \frac{\delta}{\rho} D\varphi(x)) \leq 0$.
  The positive homogeneity of $H_\infty$ yields $H_\infty(x, D\varphi(x)) \leq 0$.
  We then have
  \[
    \delta w^\rho_\delta(x) - \rho w(x) + H(x, D\varphi(x)) = -M_H + H(x, D\varphi(x))
    \leq H_\infty(x, D\varphi(x)) \leq 0 . 
  \]
  This inequality proves that $w^\rho_\delta$ is a viscosity subsolution of
  \labelcref{eq:HJ-PDE-perturbed} at any point $x$.
  Since $w$ is continuous, so is $w^\rho_\delta$, and therefore the comparison principle applies,
  leading to $w^\rho_\delta = \frac{1}{\delta} (\rho w - M_H) \leq u^\rho_\delta$.
  Similarly, we can show that $\frac{1}{\delta} (\rho w + M_H)$ is
  a viscosity supersolution of \labelcref{eq:HJ-PDE-perturbed},
  hence that $u^\rho_\delta \leq \frac{1}{\delta} (\rho w + M_H)$.

  Since $H$ is uniquely ergodic, we know that $\delta u^\rho_\delta$ converges
  to some constant $\lambda_\rho$ when $\delta$ goes to $0$.
  Thus, passing to the limit in the latter inequalities, we get
  \[
    \rho w(x) - M_H \leq \lambda_\rho  \leq \rho w(y) + M_H
  \]
  for all $x, y \in \R^n$ and all $\rho > 0$, which yields
  \[
    w(x) - w(y) \leq \frac{2 M_H}{\rho} .
  \]
  Letting $\rho$ goes to $+\infty$, we obtain that $w(x) - w(y) \leq 0$
  for all $x, y \in \R^n$, hence that $w$ is constant.
  This concludes the necessary part of the proof.

  We now prove the sufficient part.
  To that end, we assume that the structural equicontinuity property and that
  the strong maximum principle hold true.
  Let $g$ be any function in $\Cper{\R^n}$ and let us denote by $u_\delta$ the solution of
  \Cref{eq:HJ-PDE-stationary} with the Hamiltonian $g+H$.
  We have already mentioned at the beginning of the proof that the family
  $\{ \delta u_\delta \}_{0 < \delta \leq 1}$ is uniformly bounded.
  Since it is also equicontinuous by hypothesis, the Arzel\`a-Ascoli theorem entails
  the existence of a subsequence that converges uniformly to some continuous and
  $\Z^n$-periodic function $w$.

  Multiplying \labelcref{eq:HJ-PDE-stationary} by $\delta$, we get that the function
  $\delta u_\delta$ solves in $\R^n$ the equation
  \[
    \delta u(x) + \delta g(x) + \delta H\left( x, \delta^{-1} Du(x) \right) = 0 
  \]
  with $u$ being $\Z^n$-periodic.
  Since $(x,r,p) \mapsto \delta r + \delta g(x) + \delta H(x, \delta^{-1} p)$
  converges as $\delta$ goes to $0$ to $(x,r,p) \mapsto H_\infty(x, p)$ locally uniformly
  in $\R^n \times \R \times \R^n$, the stability property of viscosity solutions yields that
  the uniform limit $w$ is solution of \labelcref{eq:HJ-PDE-homogeneous},
  hence constant since the strong maximum principle applies.
  We then deduce that \Cref{itm:abstract-ergodicity-iii} of \Cref{thm:ergodicity} is satisfied.
  Indeed the implication
  \labelcref{itm:abstract-ergodicity-i}~$\Rightarrow$~\labelcref{itm:abstract-ergodicity-iii}
  remains true if, instead of the whole family $\{\delta u_\delta\}$,
  there is only a subsequence of $\{\delta u_\delta\}$ that converges uniformly to a constant
  (for the details, see the proof of \cite[Thm.~4]{AB03}).
  Thus the Hamiltonian $g+H$ is ergodic which proves that $H$ is uniquely ergodic.
\end{proof}

With a straightforward adaption of the proof, which we leave to the reader, we can also get
a sufficient condition of ergodicity with the following weaker hypothesis.

\begin{proposition}
  \label{prop:ergodicity-Hamiltonians}
  Let $H$ be an arbitrary Hamiltonian satisfying \Cref{asm:Hamiltonian}.
  If the family $\{ \delta u_\delta \}_{0 < \delta \leq 1}$, where $u_\delta$ is
  the solution of \labelcref{eq:HJ-PDE-stationary},
  is equicontinuous and if the strong maximum principle holds, then $H$ is ergodic.
  \qed
\end{proposition}

\begin{example}
  \label{ex:running-example-1}
  Consider a differential game with state space in $\R^2$ whose dynamics
  is defined for all $x \in \R^2$ by
  \[
    f(x,a,b) = \begin{pmatrix} a \\ \gamma b \end{pmatrix} ,
    \quad a,b \in [-1,1] ,
  \]
  with $0 < \gamma \leq 1$.
  Then, as we shall see in the next section (see \Cref{ex:equicontinuity}), for any
  payoff function $\ell$ satisfying \Cref{asm:standing-assumption}, the family of value
  functions $\{ \delta v_\delta \}_{0 < \delta \leq 1}$ is equicontinuous.
  On the other hand, the recession operator of the Hamiltonian of the game is
  \[
    H_\infty(x, p) = \abs{p_1} - \gamma \abs{p_2} ,
    \quad x , p = \begin{pmatrix} p_1 \\ p_2 \end{pmatrix} \in \R^2 ,
  \]
  and we know that \labelcref{eq:HJ-PDE-homogeneous} has a nonconstant solution if and only if
  $\gamma \in \Q$ (see e.g., \cite{Car10}).
  Thus, the game is ergodic if $\gamma$ is irrational.
\end{example}

\subsection{Equicontinuity of $\boldsymbol{\{ \delta u_\delta \}}$}
\label{sec:equicontinuity}

\Cref{thm:unique-ergodicity-Hamiltonians,prop:ergodicity-Hamiltonians} tell us that
(unique) ergodicity relies on two distinct properties.
As we shall see in \Cref{sec:PDE-approach}, the strong maximum principle is a qualitative
feature of the underlying dynamical system, which can be systematically characterized.
On the other hand, the (structural) equicontinuity property appears more difficult to
apprehend and is rather related with quantitative aspects of the underlying dynamics
(e.g., controllability assumptions with specific time estimates).
We next review two sufficient conditions on any Hamiltonian $H$ that guarantee
the equicontinuity of the family $\{ \delta u_\delta \}$.
Let us mention that for both conditions, the equicontinuity property is stable by perturbations
of $H$ with functions $g \in \Cper{\R^n}$, that is, equicontinuity is ``structural''
in the sense of \Cref{thm:unique-ergodicity-Hamiltonians}.

\medskip
The first of these conditions is a classic:
it is well known that equicontinuity of $\{ \delta u_\delta \}$ holds if $H$ is
coercive in the second variable, i.e., if
\[
  \lim_{|p| \to +\infty} H(x,p) = +\infty
\]
uniformly in $x$.
More precisely, this property implies that the family $\{ u_\delta \}$ is uniformly
Lipschitz continuous.
This yields in particular the existence of a corrector, that is,
a solution to \labelcref{eq:cell-problem} (see \cite{LPV87}).

\medskip
Secondly, the equicontinuity property holds if $H$ is uniformly
continuous in $x$, uniformly with respect to $p$, i.e., if
there exists a modulus of continuity $\omega$ such that
\begin{equation}
  \label{eq:equicontinuity}
  \abs{H(x,p) - H(y,p)} \leq \omega(\Norm{x-y})
\end{equation}
for all $x,y \in \R^n$ and all $p \in \R^n$.

Indeed, the equicontinuity of $\{ \delta u_\delta \}$ readily follows from
the comparison principle, after noticing that
$u_\delta(\cdot+h) - \delta^{-1} \omega(\Norm{h})$ and
$u_\delta(\cdot+h) + \delta^{-1} \omega(\Norm{h})$ are respectively subsolution
and supersolution of \labelcref{eq:HJ-PDE-stationary} (see \cite{Car10}).

\begin{example}
  Assume that $H(x,p) = \widetilde{H}(p) - \tilde{\ell}(x)$, where the function
  $\widetilde{H} : \R^n \to \R$ is continuous and $\tilde{\ell} : \R^n \to \R$ is
  continuous and $\Z^n$-periodic.
  Then $H$ satisfies \labelcref{eq:equicontinuity}, hence the structural equicontinuity
  property holds.
\end{example}

\begin{example}
  \label{ex:equicontinuity}
  Assume that $H$ is the Hamiltonian of a deterministic zero-sum differential game
  $\Gamma^-$ for which the function $f$ that controls the dynamics only depends on
  the control variables and not on the state, that is, $f(x,a,b) = \tilde{f}(a,b)$
  for some continuous function $\tilde{f} : A \times B \to \R^n$ and for all $x$.
  Then $H$ writes 
  \[
    H(x,p) = \min_{b \in B} \max_{a \in A} \big\{ -\scalar{\tilde{f}(a,b)}{p} - \ell(x,a,b) \big\}
  \]
  and one can easily see that it satisfies condition \labelcref{eq:equicontinuity}
  with modulus of continuity $\omega_\ell$.
  Thus the structural equicontinuity property holds.
  Observe that if $\ell(x,a,b) = \tilde{\ell}(x)$ for all $x,a,b$ and
  some $\tilde{\ell} \in \Cper{\R^n}$, then we recover as a special case the previous example.
\end{example}

\section{Unique ergodicity of games via PDE approach}
\label{sec:PDE-approach}

In the whole section, we fix a deterministic zero-sum differential game
in its lower form, $\Gamma^-$, which satisfies \Cref{asm:standing-assumption}.
We denote by $H$ the Hamiltonian of the game, defined in~\labelcref{eq:Hamiltonian}, and
by $H_\infty$ its recession operator \labelcref{eq:homogeneous-Hamiltonian}.

Since the values $v_\delta(\cdot)$ and $v(t,\cdot)$ of the game $\Gamma^-$ are characterized as
viscosity solutions of Hamilton-Jacobi-Isaacs PDEs (\Cref{thm:HJ-PDE-characterization}),
we can define the (unique) ergodicity of $\Gamma^-$ by applying the definitions to
its Hamiltonian $H$.
This leads to the following definition.

\begin{definition}[Ergodicity of differential games]
  The differential game $\Gamma^-$ is {\em ergodic} if the normalized value
  $\delta v_\delta(x)$ converges uniformly in $x$ to a constant when $\delta$
  goes to $0$ (or equivalently if $v(t,x) / t$ converges uniformly to a constant
  when $t$ goes to $+\infty$).

  The game $\Gamma^-$ is {\em uniquely ergodic} if for every continuous
  and $\Z^n$-periodic function $g : \R^n \to \R$, the perturbed game
  with running payoff $(x,a,b) \mapsto \ell(x,a,b) + g(x)$,
  all other data being equal, is ergodic.
\end{definition}

Thus, \Cref{thm:unique-ergodicity-Hamiltonians} or
\Cref{prop:ergodicity-Hamiltonians} already provides conditions for (unique) ergodicity.
The purpose of this section is to give other conditions, which rely on the main tool of
this article, namely {\em dominions}.
We first introduce this concept, which only rely on the controlled
system~\labelcref{eq:controlled-system}, and then use it to characterize
the (unique) ergodicity property.

\subsection{Dominions}

Informally speaking, dominions are subsets of state that can be made approximately invariant
by one player for an arbitrary period of time.
This is an adaptation to the framework of differential games of a notion that was used
to study zero-sum repeated games, played in discrete time
(see in particular the companion works \cite{AGH19,Hoc19}).
However, as we will prove in \Cref{sec:dominion-characterization}, the notion coincides
with the one of leadership domain and discriminating domain which appears in viability theory
(see, e.g., \cite{Car96}).

Before giving the formal definition of a dominion, and with the aim of simplifying the notation,
let us further mention that we shall hereafter write $\X[x,\alpha,b]{t}$, instead of
$\X[x,\alpha{[b]},b]{t}$, the solution of the controlled system~\labelcref{eq:controlled-system}
induced by a strategy $\alpha \in \NASmin$ of player~1 and a control $b \in \Bcal$ of player~2.
Also, we let $\dist{K}{x}$ be the distance of a point $x \in \R^n$ to
a subset $K \subset \R^n$, that is,
\[
  \dist{K}{x} := \inf_{y \in K} \Norm{y - x} .
\]

\begin{definition}[Dominions]
  \label{def:dominions}
  A dominion of the first player in the lower game $\Gamma^-$ is a nonempty closed set
  $D \subset \R^n$ such that for every initial position in $D$, player~1 can force
  the state to remain approximately in $D$ for any arbitrary period of time, meaning that
  \begin{multline*}
    \forall x \in D , \quad \forall \varepsilon > 0 , \quad \forall T \geq 0 , \quad
    \exists \alpha \in \NASmin , \\
    \quad \forall b \in \Bcal, \quad \forall t \in [0,T] , \quad
    \dist{D}{\X[x,\alpha,b]{t}} \leq \varepsilon .
  \end{multline*}

  Dominions for the second player are defined accordingly.
  Specifically, a dominion of player~2 in $\Gamma^-$ is a nonempty closed set
  $D \subset \R^n$ such that
  \begin{multline*}
    \forall x \in D , \quad \forall \varepsilon > 0 , \quad \forall T \geq 0 , \quad
    \forall \alpha \in \NASmin , \quad \exists b \in \Bcal , \\
    \forall t \in [0,T] , \quad \dist{D}{\X[x,\alpha,b]{t}} \leq \varepsilon .
  \end{multline*}
\end{definition}

The definition of dominions in the upper game $\Gamma^+$ is identical after switching
the identity of the players.
As we shall see in \Cref{sec:dominion-characterization}, when Isaacs' condition holds,
the definitions in the lower and the upper game coincide.

We next illustrate the notion of dominion with two examples.
In the first one, Isaacs' condition holds true, which allows us to choose for each player
the more convenient definition.
In the second example however, we provide a game for which the sets of dominions for each player
are not the same in the lower and the upper form.

\begin{example}
  \label{ex:running-example-2}
  Consider the game already introduced in \Cref{ex:running-example-1},
  whose controlled system is defined in $\R^2$ by the function
  \[
    f(x,a,b) = \begin{pmatrix} a \\ \gamma b \end{pmatrix} ,
    \quad a,b \in [-1,1] ,
  \]
  with $0 < \gamma \leq 1$ (and with any payoff function $\ell$ satisfying
  \Cref{asm:standing-assumption}).
  Let us observe that Isaacs' condition holds true for $H_\infty$:
  \begin{multline*}
    H_\infty(x,p)
    = \min_{b \in [-1,1]} \max_{a \in [-1,1]} \{ -a p_1 - \gamma b p_2 \} \\
    = \max_{a \in [-1,1]} \min_{b \in [-1,1]} \{ -a p_1 - \gamma b p_2 \} 
    = \abs{p_1} - \gamma \abs{p_2} .
  \end{multline*}
  Hence, according to \Cref{rmk:lower-upper-games} in
  \Cref{sec:dominion-characterization}, the dominions are the same in the lower and
  the upper game, and we can use for each player the simplest definition in order to
  describe them, namely
  for player~1: dominions as defined in $\Gamma^-$;
  for player~2: dominions as defined in $\Gamma^+$.
  Following this observation, we can easily see that any line of the form
  \[
    V^1_\mu = x + \begin{pmatrix} \mu \\ 1 \end{pmatrix} \R :=
    \left\{ \begin{pmatrix} x_1 + s \mu \\ x_2 + s \end{pmatrix} \right\}_{s \in \R}
  \]
  with $x \in \R^2$ and $-1 \leq \mu \gamma \leq 1$ is a dominion of player~1.
  Indeed, in the lower game, if she uses the strategy $\alpha[b] = \mu \gamma b$
  against all $b \in \Bcal$, then $V^1_\mu$ will be invariant for any initial point in it.
  Dually, any line of the form
  \[
    V^2_\nu = x + \begin{pmatrix} 1 \\ \nu \end{pmatrix} \R :=
    \left\{ \begin{pmatrix} x_1 + s \\ x_2 + s \nu \end{pmatrix} \right\}_{s \in \R}
  \]
  with $-\gamma \leq \nu \leq \gamma$ is a dominion of player~2.
  Indeed, in the upper game, he can choose the strategy $\beta[a] = \frac{\nu}{\gamma} a$
  against all $a \in \Acal$ to ensure the invariance of $V^2_\nu$.
\end{example}

\begin{example}
  \label{ex:no-isaacs}
  Contrary to the latter example, let us now illustrate the situation in which Isaacs' condition
  fails and the set of dominions of each player is not the same in the lower and the upper game.
  So, consider a differential game with state space in $\R$ whose dynamics is defined
  for all $x \in \R$ by
  \[
    f(x,a,b) = (a-b)^2 , \quad a,b \in [-1,1] ,
  \]
  and the payoff function is any continuous function that is $\Z$-periodic in $x$.
  For such a game, we have $H_\infty^-(x,p) = \max(0, -p)$ whereas
  $H_\infty^+(x,p) = \min(0, -p)$ for all $p \in \R$, which proves that Isaacs' condition
  does not hold.
  
  Then observe that in the lower game, any single point is a dominion of player~1 whereas
  the dominions of player~2 are all of the form $[x, +\infty)$.
  Symmetrically, the set of dominions of player~2 in the upper game contains any singleton,
  whereas the set of dominions for player~1 contains only intervals of the form $[x, +\infty)$.
  We further mention that the lower and the upper game are both uniquely ergodic, as we shall see
  with \Cref{thm:unique-ergodicity-PDE-approach}.
\end{example}

Before going on with ergodicity conditions, let us recall that the state space is essentially
the $n$-torus $\torus{n}$.
However, the image of a closed set in $\torus{n}$ is not necessarily closed,
which is problematic when considering dominions.
To illustrate this issue, think of the dominions $V^1_\mu$ and $V^2_\nu$ described
in \Cref{ex:running-example-2} when $\mu$ or $\nu$ are irrational, i.e.,
when their image in $\torus{2}$ is dense.
For this reason, we introduce the following definition of ``dominion in the torus''.
Note that we let $\pi : \R^n \to \torus{n}$ be the quotient map.

\begin{definition}[Dominion in the torus]
  A set $K \subset \torus{n}$ is a dominion in the torus of some player
  if $K = \clo{\pi(D)}$ for some dominion $D \subset \R^n$ of that player in $\Gamma^-$.
\end{definition}

Note that if $D \subset \R^n$ is a dominion, then
$\pi^{-1} \bigp{ \clo{\pi(D)} } = \clo{\pi^{-1}(\pi(D))}$ is also a dominion in~$\R^n$.
Furthermore, the latter set is $\Z^n$-translation-invariant, meaning that for every
$x \in \pi^{-1} \bigp{ \clo{\pi(D)} }$ and every $k \in \Z^n$, we have
$x + k \in \pi^{-1} \bigp{ \clo{\pi(D)} }$.

\subsection{Necessary condition for unique ergodicity}

We provide here a necessary condition for unique ergodicity involving dominions in the torus.
The result is based on the very simple idea that a player will leverage one of his dominion if
the payoff is more favorable on this dominion than in the rest of the states.

\begin{proposition}
  \label{prop:unique-ergodicity-only-if}
  If the differential game $\Gamma^-$ is uniquely ergodic, then the intersection of
  every dominion of player~1 with every dominion of player~2 in the torus is nonempty,
  that is, for every dominion $D^1$ of player~1 and every dominion $D^2$ of player~2
  in $\R^n$, we have
  \[
    \clo{\pi(D^1)} \cap \clo{\pi(D^2)} \neq \emptyset .
  \]
\end{proposition}

To prove this result, we will need the following technical lemmas, which give an equivalent
characterization of the dominions.
In their statement, we denote by $K_\varepsilon$ the set of points whose distance to
a subset $K \subset \R^n$ is not greater than $\varepsilon > 0$, i.e.,
\[
  K_\varepsilon := \{x \in \R^n \mid \dist{K}{x} \leq \varepsilon \} .
\]
Also, we denote by $\indic{K}$ the indicator function of $K$, defined by $\indic{K}(x) = 1$
if $x \in K$ and $\indic{K}(x) = 0$ if $x \notin K$.
Let us further recall the following standard estimate on the trajectories
of~\labelcref{eq:controlled-system} (where $\supnorm{f} = M_f$):
\begin{equation}
  \label{eq:trajectories-estimate-i}
  \Norm{\X[x,a,b]{t} - x} \leq M_{f} \, t
\end{equation}
for all $x,y \in \R^n$, $a \in \Acal$, $b \in \Bcal$ and $t \geq 0$.

\begin{lemma}
  \label{lem:dominion-characterization-i}
  A nonempty closed set $D \subset \R^n$ is a dominion of player~1
  in $\Gamma^-$ if and only if for some (hence all) $\delta > 0$,
  \[
    \forall x \in D , \quad \forall \varepsilon > 0 , \quad
    \sup_{\alpha \in \NASmin} \inf_{b \in \Bcal} \delta \int_0^{\infty} e^{-\delta s}
    \indic{D_{\varepsilon}} \big( \X[x,\alpha,b]{s} \big) ds = 1 .
  \]
\end{lemma}

\begin{proof}
  We first assume that $D$ is a dominion of player~1 and fix some discount
  factor $\delta > 0$.
  Let $x \in D$ and $\varepsilon > 0$.
  For any horizon $T > 0$, there is a strategy $\bar{\alpha} \in \NASmin$ of
  player~1 such that,
  for all controls $b \in \Bcal$ of player~2 and all times $t \in [0,T]$,
  $\X[x,\bar{\alpha},b]{t} \in D_{\varepsilon}$.
  So, for all $b \in \Bcal$ we have
  \[
    \delta \int_0^{\infty} e^{-\delta s}
    \indic{D_\varepsilon} \big(\X[x,\bar{\alpha},b]{s} \big) ds \geq
    \delta \int_0^T e^{-\delta s} ds = 1 - e^{-\delta T} .
  \]
  Hence we get
  \[
    \sup_{\alpha \in \NASmin} \inf_{b \in \Bcal} \delta \int_0^{\infty} e^{-\delta s}
    \indic{D_{\varepsilon}} \big( \X[x,\alpha,b]{s} \big) ds \geq 1 - e^{-\delta T}
  \]
  for all $T > 0$.
  Taking the limit as $T$ goes to $+\infty$, and since the integral is bounded
  above by $1$, we finally get that
  \[
    \sup_{\alpha \in \NASmin} \inf_{b \in \Bcal} \delta \int_0^{\infty} e^{-\delta s}
    \indic{D_{\varepsilon}} \big( \X[x,\alpha,b]{s} \big) ds = 1 .
  \]

  We now assume that $D$ is not a dominion of player~1.
  Since it is nonempty, it means that there exist some $\bar{x} \in D$, $\varepsilon > 0$
  and $T > 0$ such that for all strategies $\alpha$ of player~1, player~2 can choose
  a control $b$ for which $\X[\bar{x},\alpha,b]{t} \notin D_{2 \varepsilon}$
  at some $t \in [0,T]$.
  Using the estimate~\labelcref{eq:trajectories-estimate-i} we deduce that
  \[
    \forall s \in \left[ t-\frac{\varepsilon}{M_f}, t+\frac{\varepsilon}{M_f} \right] ,
    \quad \Norm{\X[\bar{x},\alpha,b]{s} - \X[\bar{x},\alpha,b]{t}} \leq \varepsilon ,
  \]
  hence $\X[\bar{x},\alpha,b]{s} \notin D_\varepsilon$.
  Note that, since $\X[\bar{x},\alpha,b]{t} \notin D_{2 \varepsilon}$, the estimate
  \labelcref{eq:trajectories-estimate-i} necessarily implies
  $t-\frac{\varepsilon}{M_f} > 0$.

  For any $\delta > 0$ we then have
  \begin{align*}
    \delta \int_0^{\infty} e^{-\delta s} \indic{D_\varepsilon}
    \big( \X[\bar{x},\alpha,b]{s} \big) ds
    & \leq \delta \int_0^{t-\frac{\varepsilon}{M_f}} e^{-\delta s} ds
    + \delta \int_{t+\frac{\varepsilon}{M_f}}^{\infty} e^{-\delta s} ds \\
    & = 1 - 2 e^{-\delta t} \sinh \left( \frac{\delta \varepsilon}{M_f} \right) \\
    & \leq 1 - 2 e^{-\delta T} \sinh \left( \frac{\delta \varepsilon}{M_f} \right) .
  \end{align*}
  Thus
  \[
    \sup_{\alpha \in \NASmin} \inf_{b \in \Bcal} \delta \int_0^{\infty} e^{-\delta s}
    \indic{D_{\varepsilon}} \big( \X[\bar{x},\alpha,b]{s} \big) ds
    \leq 1 - 2 e^{-\delta T} \sinh \left( \frac{\delta \varepsilon}{M_f} \right) < 1 ,
  \]
  which concludes the proof.
\end{proof}

With a minor adaptation of the proof, which we leave to the reader,
we can show a dual characterization of dominions for the second player.

\begin{lemma}
  \label{lem:dominion-characterization-ii}
  A nonempty closed set $D \subset \R^n$ is a dominion of player~2
  in $\Gamma^-$ if and only if for some (hence all) $\delta > 0$,
  \begin{flalign*}
    \phantom{\qed} &&
    \forall x \in D , \quad \forall \varepsilon > 0 , \quad
    \inf_{\alpha \in \NASmin} \sup_{b \in \Bcal} \delta \int_0^{\infty} e^{-\delta s}
    \indic{D_{\varepsilon}} \big( \X[x,\alpha,b]{s} \big) ds = 1 . &&
    \qed
  \end{flalign*}
\end{lemma}

\begin{remark}
  We can also give a similar characterization of dominions replacing the discounted averages
  with the time averages
  \[
    \frac{1}{T} \int_0^T \indic{D_{\varepsilon}} \big( \X[x,\alpha,b]{s} \big) ds .
  \]
\end{remark}

We can now give the proof of the necessary condition for unique ergodicity.

\begin{proof}[Proof of \Cref{prop:unique-ergodicity-only-if}]
  We prove the contrapositive and, to this end, we suppose that there exist in $\R^n$
  a dominion of player~1, denoted $D^1$, and a dominion of player~2, denoted $D^2$,
  such that $\clo{\pi(D^1)} \cap \clo{\pi(D^2)} = \emptyset$.
  Since the sets $\pi^{-1} \bigp{ \clo{\pi (D^{1 \backslash 2})} }$ are also dominions
  in $\R^n$, we can assume without loss of generality that $D^{1 \backslash 2}$ are
  $\Z^n$-translation-invariant and that $D^1 \cap D^2 = \emptyset$.
  So we can find $\varepsilon > 0$ such that $D^1_\varepsilon$ and $D^2_\varepsilon$ also
  have an empty intersection (recall that
  $D^{1 \backslash 2}_\varepsilon =
  \{ x \in \R^n \mid \dist{D^{1 \backslash 2}}{x} \leq \varepsilon \}$).
  We then consider any function $g \in \Cper{\R^n}$ satisfying 
  \begin{equation}
    \label{eq:perturbation}
    \begin{cases}
    g(x) = 0 , \quad & \forall x \in D^1_\varepsilon , \\
    g(x) = 3 M_\ell , \quad & \forall x \in D^2_\varepsilon , \\
    0 \leq g(x) \leq 3 M_\ell , \quad & \text{otherwise} ,
    \end{cases}
  \end{equation}
  where $M_\ell$ equals $\supnorm{\ell}$ if $\ell \neq 0$ and any positive real otherwise.
  Thus, the function $g$ satisfies, for all $x \in \R^n$,
  \[
    3 M_\ell \indic{D^2_\varepsilon}(x) \leq g(x) \leq
    3 M_\ell (1 - \indic{D^1_\varepsilon}(x) ) .
  \]
  
  Let $\delta > 0$ be any discount factor.
  From the above inequalities, we deduce that for all $x \in \R^n$,
  all strategies $\alpha$ of player~1 and all controls $b$ of player~2,
  \begin{multline*}
    - M_\ell + \delta \int_0^{\infty} 3 M_\ell e^{-\delta s}
    \indic{D^2_{\varepsilon}} \big( \X[x,\alpha,b]{s} \big) ds \\
     \leq \delta \int_0^{\infty} e^{-\delta s}
    \Big( \ell(\X[x,\alpha,b]{s},\alpha[b]_s,b_s) + g(\X[x,\alpha,b]{s}) \Big) ds \\
    \leq M_\ell + \delta \int_0^{\infty} 3 M_\ell e^{-\delta s}
    \big(1 - \indic{D^1_{\varepsilon}} \big( \X[x,\alpha,b]{s} \big) \big) ds .
  \end{multline*}
  Let us denote by $v_\delta^{g}$ the unnormalized value of the discounted game
  with the perturbed running payoff $(x,a,b) \mapsto \ell(x,a,b) + g(x)$.
  Taking the supremum over $b \in \Bcal$ and then the infimum over $\alpha \in \NASmin$
  in the latter inequalities, we deduce from \Cref{lem:dominion-characterization-i} that
  $\delta v_\delta^g(x) \leq M_\ell$ for all $x \in D^1$, and from
  \Cref{lem:dominion-characterization-ii} that $2 M_\ell \leq \delta v_\delta^g(y)$
  for all $y \in D^2$.
  Thus, if $x \in D^1$ and $y \in D^2$, we have
  \[
    \liminf_{\delta \to 0} \big( \delta v_\delta^g(y) - \delta v_\delta^g(x) \big) \geq M_\ell > 0
  \]
  which proves that the perturbed game is not ergodic, hence that the game $\Gamma^-$ is not
  uniquely ergodic. 
\end{proof}

\begin{remark}[Comparison with one-player controlled systems]
  It is instructive to compare the latter necessary condition of unique ergodicity with
  the result of Arisawa in \cite{Ari97}, which deals with optimal control problems, i.e.,
  problems for systems controlled by one player (who is minimizing and which we call player~1).
  In this paper, she proved that if the controlled system is uniquely ergodic, then
  there exists an {\em ergodic attractor} $D$ which satisfies the following properties.
  \begin{enumerate}
    \item[(P)] $D$ is closed, connected and positively invariant.
    \item[(D)] $D$ is nonempty and $y \in D$ if and only if for any $x \in \R^n$ and
      any $\varepsilon > 0$, there exists $T_\varepsilon > 0$ and $a_\varepsilon \in \Acal$
      such that $\lim_{\varepsilon \to 0} T_\varepsilon = +\infty$ and
      $\Norm{y -\X[x,a_{\varepsilon}]{T_\varepsilon}} < \varepsilon$.
    \item[(A)] $D$ has the following time-averaged attracting property:
      for any neighborhood $U$ of $D$ and any $x \in \R^n$,
      \[
        \lim_{\delta \to 0} \delta \int_0^\infty e^{-\delta s} \indic{U}(\X[x,a]{s}) ds = 1 ,
        \qquad \text{uniformly in $a \in \Acal$.}
      \]
  \end{enumerate}

  For such controlled systems, if we introduce a second player as a dummy to cast the problem
  within the framework of two-player differential games, then it readily follows from
  the definition that the dominions of player~2 correspond to the nonempty closed and
  positively invariant sets (indeed, every positive orbit through any point in a dominion of
  player~2 is within any $\varepsilon$-neighborhood of the dominion
  for any arbitrary period of time).
  Let us observe that these sets are also dominions of player~1 and that the intersection of
  two dominions of player~2, if nonempty, is another dominion of player~2.

  Then, applying \Cref{prop:unique-ergodicity-only-if}, we deduce that
  if unique ergodicity holds, there is a unique minimal nonempty closed positively invariant set
  in the torus and that this set intersect every dominion of player~1 in the torus.
  We claim that this set is the ergodic attractor $D$ described in \cite{Ari97} and that
  the two results are equivalent.
  Indeed it follows from the properties (P) and (D) that the ergodic attractor $D$ is
  the unique minimal dominion of player~2 (the uniqueness comes from the connectedness
  in (P) and the minimality from (D)) and property (A) implies that any dominion of
  player~1 cannot be disjoint from $D$.
  Conversely, if $D$ is the unique minimal dominion of player~2 whose existence stems
  from \Cref{prop:unique-ergodicity-only-if}, then property (P) is readily verified.
  Furthermore, its minimality implies that any point $x \in D$ is approximately
  controllable to any other point $y \in D$.
  Then, since every dominion of player~1 meets $D$, and particularly the closure of
  any positive orbit, we can show that property (D) holds.
  Finally using (P) and (D) we can then prove that (A) holds, as is done in \cite{Ari97}.
\end{remark}

\subsection{Sufficient condition for unique ergodicity}

In this subsection, we give a sufficient condition of unique ergodicity which is derived
from \Cref{thm:unique-ergodicity-Hamiltonians}.
We start with a lemma that relates the solutions of \labelcref{eq:HJ-PDE-homogeneous}
to dominions in $\Gamma^-$.

\begin{lemma}
  \label{lem:argmin-max}
  Let $w$ be any continuous viscosity solution of \labelcref{eq:HJ-PDE-homogeneous}.
  Then $\argmin w$ is a dominion of player~1 in $\Gamma^-$ and $\argmax w$ is
  a dominion of player~2.
\end{lemma}

\begin{proof}
  Let us first consider the differential game with the same definition as $\Gamma^-$
  except for the payoff function $\ell$ which is replaced with $w$.
  The Hamiltonian associated to this game is $H_\infty(x,p)-w(x)$ and since, for
  any $\delta > 0$, the function $\delta^{-1} w$ is solution to
  \labelcref{eq:HJ-PDE-stationary} with the latter Hamiltonian, we deduce from
  \Cref{thm:HJ-PDE-characterization} that it is the (unnormalized) value of
  the infinite-horizon discounted game.
  Thus, for all points $x$ in $\R^n$ and all positive factors $\delta$ we have
  \begin{equation}
    \label{eq:argmin-max}
    \delta^{-1} w(x) = \inf_{\alpha \in \NASmin} \sup_{b \in \Bcal}
    \int_0^{\infty} e^{-\delta s} w(\X[x,\alpha,b]{s}) ds .
  \end{equation}

  Now set $D = \argmin w$ and let us assume, without loss of generality, that $\min w = 0$.
  Also, since the case with $w$ constant is trivial, we can assume that $D \neq \R^n$.
  In view of \Cref{lem:dominion-characterization-i}, we fix arbitrary positive constants
  $\varepsilon$ and $\delta$.
  Again, if $D_\varepsilon = \{x \in \R^n \mid \dist{D}{x} \leq \varepsilon\}$ is
  the whole space $\R^n$, then the equality in \Cref{lem:dominion-characterization-i}
  trivially holds, so we assume that $\varepsilon$ is small enough
  so that $D_\varepsilon \neq \R^n$.
  Then, denoting by $m_\varepsilon$ the infimum of $w$ on the complement of
  $D_\varepsilon$, which is necessarily positive, we can write
  \[
    w(x) \geq m_\varepsilon \left( 1 - \indic{D_\varepsilon}(x) \right)
  \]
  for all $x \in \R^n$.
  By plugging this inequality into the right-hand side of \labelcref{eq:argmin-max},
  we obtain for all $x \in \R^n$
  \[
    w(x) \geq \inf_{\alpha \in \NASmin} \sup_{b \in \Bcal} \delta \int_0^{\infty}
    e^{-\delta s} m_\varepsilon
    \left( 1 - \indic{D_\varepsilon}(\X[x,\alpha,b]{s}) \right) ds .
  \]
  After simplification, this yields, for all $x \in D$,
  \[
    0 \geq 1 - \sup_{\alpha \in \NASmin} \inf_{b \in \Bcal} \delta \int_0^{\infty}
    e^{-\delta s} \indic{D_\varepsilon}(\X[x,\alpha,b]{s}) ds .
  \]
  Since the converse inequality is obviously true, we deduce that there is in fact equality
  and thus, by \Cref{lem:dominion-characterization-i}, that $D$ is a dominion of player~1.

  With very similar arguments and using \Cref{lem:dominion-characterization-ii}
  instead of \Cref{lem:dominion-characterization-i}, we can show that
  $\argmax w$ is a dominion of player~2.
\end{proof}

We know that if the value function $\delta v_\delta$ converges uniformly to some
function $v$ then it is solution to \labelcref{eq:HJ-PDE-homogeneous}.
This entails the following corollary

\begin{corollary}
  Assume that the value function $\delta v_{\delta}$ of the game $\Gamma^-$ converges
  uniformly to some function $v$.
  Then $\argmin v$ and $\argmax v$ are dominions of player~1 and player~2, respectively.
\end{corollary}

A straightforward consequence is that if $\argmin v$ and $\argmax v$ have a nonempty
intersection, then $v$ is constant and the game is ergodic.
We can extend this result to unique ergodicity with the help of
\Cref{thm:unique-ergodicity-Hamiltonians} and thus provide a converse
to \Cref{prop:unique-ergodicity-only-if}.

\begin{theorem}
  \label{thm:unique-ergodicity-PDE-approach}
  Assume that in the differential game $\Gamma^-$, the intersection of every dominion
  of player~1 with every dominion of player~2 in the torus is nonempty.
  Then the strong maximum principle (see \Cref{thm:unique-ergodicity-Hamiltonians}) holds,
  i.e., the constant functions are the only solutions of \labelcref{eq:HJ-PDE-homogeneous}.

  If, moreover, the structural equicontinuity property is true, then $\Gamma^-$ is uniquely
  ergodic if and only if the two players do not have disjoint dominions in the torus.
\end{theorem}

\begin{proof}
  Let $w$ be any solution of \labelcref{eq:HJ-PDE-homogeneous}.
  Let $D^1 = \argmin w$ and $D^2 = \argmax w$.
  Since $w$ is $\Z^n$-periodic and continuous, it passes to the quotient into a continuous map
  on the torus whose minimum (resp., maximum) is attained on $\pi(D^1)$ (resp., $\pi(D^2)$).
  Hence, $\pi(D^{1 \backslash 2})$ are necessarily closed and we have
  $D^{1 \backslash 2} = \pi^{-1} \bigp{ \clo{\pi (D^{1 \backslash 2})} }$.
  Using now \Cref{lem:argmin-max}, we deduce that $\clo{\pi(D^1)} \cap \clo{\pi(D^2)}$
  hence $D^1 \cap D^2$ is nonempty.
  So $w$ is constant.

  The rest of the proof follows from \Cref{prop:unique-ergodicity-only-if} and
  \Cref{thm:unique-ergodicity-Hamiltonians}.
\end{proof}

Note that if the controlled system~\labelcref{eq:controlled-system} is Lipschitz
continuous, meaning that there is a positive constant $L$ for which
\[
  \Norm{\X[x,a,b]{t} - \X[y,a,b]{t}} \leq L \Norm{x-y} , \quad
  \forall x, y \in \R^n , \enspace \forall a \in \Acal , \enspace \forall b \in \Bcal ,
  \enspace \forall t \geq 0 ,
\]
then the family $\{ \delta v_\delta \}$ is equi-Lipschitz
for any payoff function $\ell$.
In that case we can use the latter theorem to characterize unique ergodicity
in terms of dominions.
This is in particular the case if the function $f$ does not depend on the state variable.

\begin{example}
  Let us go back to the game introduced in \Cref{ex:running-example-1,ex:running-example-2},
  whose dynamics is defined in $\R^2$ by the function
  \[
    f(x,a,b) = \begin{pmatrix} a \\ \gamma b \end{pmatrix} ,
    \quad a,b \in [-1,1] ,
  \]
  with $0 < \gamma \leq 1$ and whose payoff function $\ell$ is any function satisfying
  \Cref{asm:standing-assumption}.
  We already mentioned that the family of value functions
  $\{\delta v_{\delta}\}_{0 < \delta \leq 1}$ is equicontinuous
  (see \Cref{ex:equicontinuity} or the above remark).
  Hence the structural equicontinuity property holds.

  If $\gamma$ is a rational number then, for any $x, y \in \R^2$, the lines 
  \[
    V^1_{1/\gamma} = x + \begin{pmatrix} 1/\gamma \\ 1 \end{pmatrix} \R
    \qquad \text{and} \qquad
    V^2_\gamma = y + \begin{pmatrix} 1 \\ \gamma \end{pmatrix} \R
  \]
  are dominions of player~1 and player~2, respectively, and their quotient images
  in the torus $\torus{2}$ are closed and disjoint for suitable $x$ and $y$.
  Thus, according to \Cref{thm:unique-ergodicity-PDE-approach},
  the game is not uniquely ergodic.

  Assume now that $\gamma$ is not a rational number and consider in $\R^2$
  any dominions $D^1$ and $D^2$ of player~1 and player~2, respectively.
  We next show that their intersection in the torus is not empty.
  Let us fix two points, $x \in D^1$ and $y \in D^2$, in these dominions.
  By definition, given $\varepsilon > 0$ and $T > 0$, player~1 has a strategy
  $\alpha_\varepsilon \in \NASmin$ such that for every action $b \in \Bcal$ of player~2,
  we have $\dist{D^1}{\X[x,\alpha_\varepsilon,b]{t}} \leq \varepsilon$ for all $t \in [0,T]$.
  In particular, if $b$ is the constant control equal to $1$, then we have
  \[
    \begin{pmatrix} -1 \\ \gamma \end{pmatrix}
    \leq \dot{X}_{t}^{x,\alpha_\varepsilon,b} 
    \leq \begin{pmatrix} 1 \\ \gamma \end{pmatrix} ,
    \quad \forall t \geq 0 ,
  \]
  that is, the (continuous) trajectory of the dynamical system has the property that
  $\X[x,\alpha_\varepsilon,b]{t} - \X[x,\alpha_\varepsilon,b]{s}$ is included in the cone
  $C_\gamma^1 = \{ z = (z_1, z_2)^\intercal \in \R^2 \mid -z_2 \leq \gamma z_1 \leq z_2 \}$
  for all $0 \leq s \leq t \leq T$ (see \Cref{fig:orbits}).
  Likewise, with the same $\varepsilon$ and $T$, player~2 has a strategy
  $\beta_\varepsilon \in \NASmax$ such that 
  $\dist{D^2}{\X[y,a,\beta_\varepsilon]{t}} \leq \varepsilon$ for all $t \in [0,T]$ and
  all $a \in \Acal$, and if player~1 chooses the constant control equal to $1$, then we have
  \[
    \begin{pmatrix} 1 \\ -\gamma \end{pmatrix}
    \leq \dot{X}^{y,a,\beta_\varepsilon}_{t} 
    \leq \begin{pmatrix} 1 \\ \gamma \end{pmatrix} ,
    \quad \forall t \geq 0 ,
  \]
  that is, the trajectory of the system is such that
  $\X[y,a, \beta_\varepsilon]{t} - \X[y,a, \beta_\varepsilon]{s}$ is included in the cone
  $C_\gamma^2 = \{ z = (z_1, z_2)^\intercal \in \R^2 \mid -\gamma z_1 \leq z_2 \leq \gamma z_1 \}$
  for all $0 \leq s \leq t \leq T$ (see \Cref{fig:orbits}).
  
  \begin{figure}[ht]
    \centering
    \begin{tikzpicture}
      \begin{axis}[
          axis lines=none,
          xmin=-4,
          xmax=4,
          ymin=-2.5,
          ymax=2.5,
          smooth,
          unit vector ratio=1 1 1,
          scale=1.25,
        ]
        \draw[->, >= angle 60,gray] (-2.25,0) -- (2.25,0);
        \draw[black] (0,0) node[below left] {\footnotesize $0$};
        \draw[black] (2.5,0) node[below left] {\small $x_1$};
        \draw[gray,thick] (1,0) -- (1,-0.05);
        \draw[black] (1,0) node[below right] {\footnotesize $1$};
        \draw[gray,densely dotted] (-2.25,1) -- (2.25,1);
        \draw[gray,densely dotted] (-2.25,2) -- (2.25,2);
        \draw[gray,densely dotted] (-2.25,-1) -- (2.25,-1);
        \draw[->, >= angle 60,gray] (0,-1.75) -- (0,2.25);
        \draw[black] (0,2.25) node[left] {\small $x_2$};
        \draw[gray,thick] (-0.05,1) -- (0,1);
        \draw[black] (0,1) node[above left] {\footnotesize $1$};
        \draw[gray,thick] (-0.05,0.66) -- (0,0.66);
        \draw[black] (0,0.66) node[left] {\footnotesize $\gamma$};
        \draw[gray,densely dotted] (2,-1.75) -- (2,2.15);
        \draw[gray,densely dotted] (1,-1.75) -- (1,2.25);
        \draw[gray,densely dotted] (-1,-1.75) -- (-1,2.25);
        \draw[gray,densely dotted] (-2,-1.75) -- (-2,2.25);
        \draw[black!0,] (4,-1.75) -- (4,2.25);
        \draw[black!0,] (-4,-1.75) -- (-4,2.25);
        \addplot[black,thick,domain=0:2.25] {0.66*x};
        \draw[black] (-1.66,1.5) node {$C_\gamma^1$};
        \addplot[black,thick,domain=-2.25:2.25] {-0.66*x};
        \draw[black] (1.75,-0.75) node {$C_\gamma^2$};
        \addplot[black,densely dashed,domain=0:2.25] {1.33*x+0.1*sin(5*deg(x)};
        \draw[black] (1.75,2.35) node[right] {\small $\X[x,\alpha_\varepsilon,b]{t} - x$};
        \addplot[black,densely dashed,domain=0:2.25] {0.4*x-0.05*sin(4*deg(x))};
        \draw[black] (2.25,1.05) node[right] {\small $\X[y,a,\beta_\varepsilon]{t} - y$};
      \end{axis}
    \end{tikzpicture}
    \caption{Trajectories in $\varepsilon$-neighborhoods ($D^1_\varepsilon$ and $D^2_\varepsilon$)
    of dominions $D^1$ and $D^2$.}
    \label{fig:orbits}
  \end{figure}
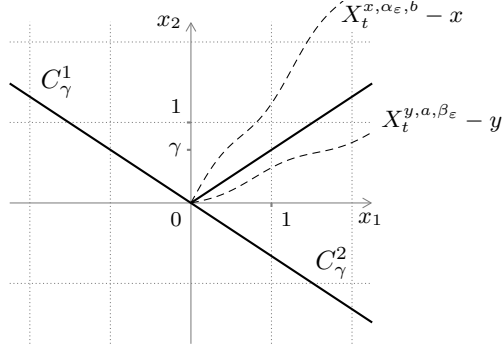

  Then, the parameter $\varepsilon$ being fixed, either there is some time $T$
  such that the images in the torus $\torus{2}$ of the two trajectories mentioned above
  intersect on the time interval $[0,T]$ at some point $z_\varepsilon \in \torus{2}$,
  or for all times $T$ their images always remain disjoint, which is possible only if
  they are contained in the parallel half-lines starting in $x$ and $y$, respectively,
  and directed by the vector $(1, \gamma)^\intercal$.
  Indeed, since $\gamma \notin \Q$, the images of these half-lines in the torus are dense,
  and therefore any deviation of a trajectory from one of these half-lines eventually
  leads to the intersection of the two trajectories.

  If there are only finitely many points $z_\varepsilon$ as described above, then
  we deduce that $D^1$ and $D^2$ respectively contain the latter half-lines and
  therefore both dominions correspond to the trivial dominion in $\torus{2}$,
  composed of the whole state space.
  If there are infinitely many points $z_\varepsilon$, then any limit point
  is, by construction, contained in both $D^1$ and $D^2$.
  In any case, we deduce that the players do not have disjoint dominions in the torus
  and so, according to \Cref{thm:unique-ergodicity-PDE-approach},
  that the game is uniquely ergodic.
\end{example}

\section{Unique ergodicity of games via controllability approach}
\label{sec:controllability-approach}

In this section, as usual, we fix a deterministic zero-sum differential game
in its lower form, $\Gamma^-$, which satisfies \Cref{asm:standing-assumption}.
However, we assume that the controlled system~\labelcref{eq:controlled-system} is not
Lipschitz continuous (and in particular that $L_f > 0$), so that equicontinuity
of $\{ \delta v_\delta \}$ cannot be guaranteed.
We also make the standard assumption that the payoff function $\ell$ is Lipschitz continuous
in $x$ uniformly in $(a,b)$, i.e., that there exists $L_\ell > 0$ such that
\[
  \abs{\ell(x,a,b)-\ell(y,a,b)} \leq L_\ell \Norm{x-y} ,
  \quad \forall x, y \in \R^n , \enspace \forall a \in A ,
  \enspace \forall b \in B .
\]
We then have the following classical regularity property of the value function.

\begin{proposition}[see {\cite[Ch.~VIII, Prop.~1.8]{BCD97}}]
  \label{prop:regularity-values}
  If $\ell$ is Lipschitz continuous in $x$ uniformly in $(a,b)$, then, for any discount
  factor $\delta < L_f$, the value function $\delta v_\delta$ is H\"older continuous
  with exponent $\delta / L_f$ and constant $L$ independent of $\delta$:
  \begin{equation*}
    \label{eq:Holder-continuity}
    \Norm{\delta v_\delta(x) - \delta v_\delta(y)} \leq L \Norm{x-y}^{\delta / L_f} , \quad
    \forall x , y \in \R^n .
  \end{equation*}
\end{proposition}

In view of unique ergodicity, the requirement that the payoff function $\ell$ be uniformly
Lipschitz continuous in $x$ prevents us from considering perturbations $g$
that only lie in $\Cper{\R^n}$.
If we want to use the latter proposition (which we need to prove the main theorem of
this section), we need to restrict the perturbations $g$ to the set of Lipschitz continuous
and $\Z^n$-periodic functions.
Fortunately, this is not a major restriction.
Indeed, in the proof of \Cref{prop:unique-ergodicity-only-if} it is possible to
consider a perturbation function $g$ satisfying \labelcref{eq:perturbation} and which is Lipschitz. 
Thus we have the following stronger result.

\begin{proposition}
  \label{prop:ergodicity-only-if-Lipschitz}
  Assume that for every Lipschitz continuous and $\Z^n$-periodic function $g$ from $\R^n$
  to $\R$, the perturbed differential game with payoff function
  $(x,a,b) \mapsto \ell(x,a,b) + g(x)$ is ergodic.
  Then, in the game $\Gamma^-$, the players do not have disjoint dominions in $\torus{n}$.
  \qed
\end{proposition}

To compensate the lack of equicontinuity of $\{ \delta v_\delta \}$ we also need to
introduce the following controllability assumption, which involves sets of points that are
reachable by one player.
Let us first describe precisely these sets.

Given any strategy $\alpha \in \NASmin$ of player~1, we define the {\em reachable set}
from a point $x \in \R^n$ for player~2 by
\[
  R_\alpha^1(x) := \big\{ \X[x,\alpha,b]{t} \mid b \in \Bcal, \enspace t \geq 0 \big\} .
\]
On the other hand, for all strategies $\alpha \in \NASmin$ of player~1, let us associate
a control $b_\alpha \in \Bcal$ of player~2.
Then, we define the {\em reachable set} from $x \in \R^n$ for player~1 by
\[
  R_{b \argt}^2(x) := \big\{ \X[x,\alpha,b_\alpha]{t} \mid \alpha \in \NASmin,
\enspace t \geq 0 \big\}.
\]
Furthermore, we say that the map $\alpha \mapsto b_\alpha$ is nonanticipating if
$\alpha^1[b]_s = \alpha^2[b]_s$ for all $b \in \Bcal$ and almost all $s \in [0,t]$
implies that $(b_{\alpha^1})_s = (b_{\alpha^2})_s$ for almost all $s \in [0,t]$.
That is, if $\alpha^1$ and $\alpha^2$ coincide almost surely on $[0,t]$, then
the same is true for $b_{\alpha^1}$ and $b_{\alpha^2}$.

We emphasize that the purpose of the following assumption is only to provide a uniform bound
on the time needed to get arbitrarily close to any reachable point.
We further mention that the estimate is borrowed from \cite{Ari98} (see also \cite{Bet05}).

\begin{assumption}[Uniform time estimate]
  \label{asm:controllability}
  There exist constants $\gamma \in [0,1)$ and $C > 0$ such that, for all $\varepsilon > 0$,
  \begin{itemize}
    \item for all $\alpha \in \NASmin$, for all $x \in \R^n$ and all $y \in \clo{R^1_\alpha}(x)$,
      there is a control $b \in \Bcal$ and a time $t \leq C(-\log \varepsilon)^\gamma$ for which
      $\norm{y - \X[x,\alpha,b]{t}} \leq \varepsilon$;
    \item for all nonanticipating map $\NASmin \to \Bcal, \alpha \mapsto b_\alpha$,
      for all $x \in \R^n$ and all $y \in \clo{R^2_{b \argt}}(x)$,
      there is a strategy $\alpha \in \NASmin$ and a time $t \leq C(-\log \varepsilon)^\gamma$
      for which $\norm{y - \X[x,\alpha,b_\alpha]{t}} \leq \varepsilon$.
  \end{itemize}
\end{assumption}

We can now give the condition for the (somewhat modified version of) unique ergodicity of
differential games.
Notice that in the proof of this result, we use the fact that the sets
$\clo{R^1_\alpha}(x)$ and $\clo{R^2_{b \argt}}(x)$ are dominions of player~1 and
player~2, respectively.
We postpone the precise statement and the proof of this fact afterward.

\begin{theorem}
  \label{thm:unique-ergodicity-control-approach}
  In the differential game $\Gamma^-$, suppose that \Cref{asm:controllability} holds
  and that the payoff function $\ell$ is Lipschitz continuous in $x$ uniformly in $(a,b)$.
  The following assertions are equivalent:
  \begin{enumerate}
    \item for every function $\ell': \R^n \times A \times B \to \R$ which is
      Lipschitz continuous in $x$ uniformly in $(a,b)$ and $\Z^n$-periodic in $x$,
      the modified game with running payoff $\ell'$ is ergodic;
      \label{itm:uniq-ergod-i}
    \item for every Lipschitz continuous and $\Z^n$-periodic function $g: \R^n \to \R$,
      the perturbed game with running payoff $(x,a,b) \mapsto \ell(x,a,b) + g(x)$ is ergodic;
      \label{itm:uniq-ergod-ii}
    \item the players do not have disjoint dominions in the torus.
      \label{itm:uniq-ergod-iii}
  \end{enumerate}
\end{theorem}

\begin{proof}
  The implication \labelcref{itm:uniq-ergod-i}~$\Rightarrow$~\labelcref{itm:uniq-ergod-ii}
  is trivial and we already know from \Cref{prop:ergodicity-only-if-Lipschitz} that
  \labelcref{itm:uniq-ergod-ii}~$\Rightarrow$~\labelcref{itm:uniq-ergod-iii}.
  So we only need to prove that
  \labelcref{itm:uniq-ergod-iii}~$\Rightarrow$~\labelcref{itm:uniq-ergod-i}.
  And since the payoff function $\ell$ is arbitrary and assertion \labelcref{itm:uniq-ergod-iii}
  does not depend on it, if we prove that $\Gamma^-$ is ergodic, the result will be true
  for any other payoff function $\ell'$.

  Let $\delta > 0$ be any discount factor and let $\varepsilon$ be a fixed positive real.
  Let $x,y$ be any points in $\R^n$.
  From the dynamic programming principle, there exists a strategy $\alpha^{\! 1} \in \NASmin$
  of player~1 (which depends only on $\delta$, $\varepsilon$ and $x$) such that
  \begin{equation}
    \label{eq:DPP-1}
    v_\delta(x) + \varepsilon \geq
    \int_0^t e^{-\delta s} \ell(\X[x,\alpha^{\! 1},b]{s}, \alpha^{\! 1}[b]_s, b_s) ds
      + e^{- \delta t} v_\delta(\X[x,\alpha^{\! 1}, b]{t})
  \end{equation}
  for all times $t > 0$ and all controls $b \in \Bcal$.
  Similarly, for all $\alpha \in \NASmin$, there exists a control $b_\alpha \in \Bcal$
  of player~2 (which depends only on $\delta$, $\varepsilon$, $y$ and $\alpha$) such that
  \begin{equation}
    \label{eq:DPP-2}
    v_\delta(y) - \varepsilon \leq
    \int_0^t e^{-\delta s} \ell(\X[y,\alpha,b_\alpha]{s}, \alpha[b_\alpha]_s, (b_\alpha)_s) ds
      + e^{- \delta t} v_\delta(\X[y,\alpha, b_\alpha]{t})
  \end{equation}
  for all times $t > 0$.
  Furthermore, the map $\alpha \mapsto b_\alpha$ can be chosen nonanticipating,
  as defined above (indeed, for the controls $b_{\alpha}$ to satisfy
  these conditions, we can chose them so that
  $v_\delta(x) - \varepsilon \leq J_\delta(x,\alpha,b_\alpha$).

  Let $D^1 = \clo{R_{\alpha^{\! 1}}^1}(x)$ and $D^2 = \clo{R_{b \argt}^2}(y)$ be the closures
  of the sets of reachable points from $x$ and $y$ by player~2 and player~1, respectively,
  being fixed the strategy $\alpha^{\! 1}$ and the nonanticipating map
  $\alpha \mapsto b_\alpha$.

  We know from subsequent \Cref{lem:reachable-sets} that these sets are respectively
  a dominion of player~1 and a dominion of player~2.
  Hence there exists a point
  $z \in \pi^{-1} \bigp{ \clo{\pi(D^1)} \cap \clo{\pi(D^2)} }$.
  This implies that there are $z^1 \in D^1$, $z^2 \in D^2$ and $k, l \in \Z^n$ such that
  \[
    \Norm{z + k - z^1} \leq \varepsilon / 2 \qquad \text{and} \qquad
    \Norm{z + l - z^2} \leq \varepsilon / 2 .
  \]
  Moreover, \Cref{asm:controllability} guarantees the existence of a control
  $b^{1} \in \Bcal$, a strategy $\alpha^{\! 2} \in \NASmin$ and
  times $t_{1}, t_{2} \leq C (-\log \varepsilon)^\gamma$ such that
  \[
    \Norm{z^1 - \X[x,\alpha^{\! 1},b^{1}]{t_{1}}} \leq \varepsilon / 2
    \qquad \text{and} \qquad
    \Norm{z^2 - \X[y,\alpha^{\! 2},b_{\alpha^{\! 2}}]{t_{2}}} \leq \varepsilon / 2 .
  \]
  Combining these inequalities, we get
  \[
    \Norm{z+k - \X[x,\alpha^{\! 1},b^{1}]{t_{1}}} \leq \varepsilon
    \qquad \text{and} \qquad
    \Norm{z+l - \X[y,\alpha^{\! 2},b_{\alpha^{\! 2}}]{t_{2}}} \leq \varepsilon .
  \]

  Since the inequalities \labelcref{eq:DPP-1,eq:DPP-2} hold uniformly
  in $t$, we can now write them at times $t_{1}$ and $t_{2}$ respectively, and then use
  the estimates that we have just established.
  We recall that, for $\delta$ small enough, the function $\delta v_\delta$ is H\"older
  continuous with exponent $\delta / L_f$ and  constant $L$.
  We also recall that $v_\delta$ is $\Z^n$-periodic.
  Let $T_\varepsilon = C (-\log \varepsilon)^\gamma$ for simplicity.
  From \labelcref{eq:DPP-1} we get
  \begin{align*}
    & \delta v_\delta(x) - \delta v_\delta(z) + \delta \varepsilon \\
    & \qquad \geq -M_{\ell} (1 - e^{-\delta t_{1}})
      + e^{-\delta t_{1}}
      \big( \delta v_\delta(\X[x,\alpha^{\! 1},b^{1}]{t_{1}}) - \delta v_\delta(z+k) \big)
      - (1 - e^{-\delta t_{1}}) \delta v_\delta(z) \\
    & \qquad \geq -(1- e^{-\delta t_{1}}) (\delta v_\delta(z) + M_{\ell})
      - e^{-\delta t_{1}} L \varepsilon^{\delta / L_f} \\
    & \qquad \geq -(1- e^{-\delta T_\varepsilon}) (\delta v_\delta(z) + M_{\ell})
      - L \varepsilon^{\delta / L_f} ,
  \end{align*}
  where we use the fact that $\delta v_\delta(z) + M_{\ell} \geq 0$ and $e^{-\delta t_{1}} \leq 1$.

  On the other hand, from \labelcref{eq:DPP-2} we get
  \begin{align*}
    & \delta v_\delta(y) - \delta v_\delta(z) - \delta \varepsilon \\
    & \qquad \leq M_{\ell} (1 - e^{-\delta t_{2}})
      + e^{-\delta t_{2}}
      \big( \delta v_\delta(\X[y,\alpha^{\! 2},b_{\alpha^{\! 2}}]{t_{2}})
      - \delta v_\delta(z+l) \big)
      - (1 - e^{-\delta t_{2}}) \delta v_\delta(z) \\
    & \qquad \leq (1- e^{-\delta t_{2}}) (-\delta v_\delta(z) + M_{\ell}) +
      e^{-\delta t_{2}} L \varepsilon^{\delta / L_f} \\
    & \qquad \leq (1- e^{-\delta T_\varepsilon}) (-\delta v_\delta(z) + M_{\ell}) +
      L \varepsilon^{\delta / L_f} .
  \end{align*}
  Here we use the fact that $-\delta v_\delta(z) + M_{\ell} \geq 0$ and $e^{-\delta t_{2}} \leq 1$.

  Combining the two inequalities and letting $M = \max \{ M_\ell, L \}$,
  we obtain, for all $\delta, \varepsilon > 0$,
  \begin{equation}
    \label{eq:uniform-bound}
    \delta v_\delta(x) - \delta v_\delta(y)
    \geq -2 \delta \varepsilon - 2 M \big( 1- e^{-\delta T_\varepsilon}
    + e^{(\delta \log \varepsilon) / L_f} \big) .
  \end{equation}
  Since $T_\varepsilon = C (-\log \varepsilon)^\gamma$, choosing $\varepsilon$ such that
  $\log \varepsilon = -\delta^{-(1+\omega)}$ with $0 < \omega < \frac{1}{\gamma} - 1$,
  we observe that the right-hand side of the latter inequality converges to zero
  as $\delta$ vanishes, which yields
  \[
    \liminf_{\delta \to 0} \big( \delta v_\delta(x) - \delta v_\delta(y) \big) \geq 0 .
  \]
  Since the points $x$ and $y$ are arbitrary and the bound in \labelcref{eq:uniform-bound}
  does not depend on them, we deduce that
  \[
    \lim_{\delta \to 0} \big( \delta v_\delta(x) - \delta v_\delta(y) \big) = 0
  \]
  uniformly in $x,y \in \R^n$.

  The rest of the proof is classical (see for instance \cite{Ari98}), but one may also
  notice that the latter uniform limit together with the continuity of
  $(\delta,x) \mapsto \delta v_\delta(x)$ on $(0,1] \times \R^n$ (see the proof of
  \Cref{thm:unique-ergodicity-Hamiltonians}) entails the equicontinuity of
  the family $\{ \delta v_\delta \}_{0 < \delta \leq 1}$.
  We can then conclude with
  \Cref{prop:ergodicity-Hamiltonians,thm:unique-ergodicity-PDE-approach}.
\end{proof}

In order to complete the proof and conclude the section, we prove the following.

\begin{lemma}
  \label{lem:reachable-sets}
  Given a strategy $\alpha \in \NASmin$ of player~1, the topological closure of
  the reachable set from any point $x \in \R^n$ for player~2, $\clo{R_\alpha^1}(x)$,
  is a dominion of player~1.

  Dually, given a map $\alpha \mapsto b_\alpha$ from $\NASmin$ to $\Bcal$ which is
  nonanticipating, the closure of the reachable set from $x$ for player~1,
  $\clo{R_{b \argt}^2}(x)$, is a dominion of player~2.
\end{lemma}

\begin{proof}
  We show in detail that $\clo{R_\alpha^1}(x)$ is a dominion of player~1, and leave
  to the reader the details of the proof for $\clo{R_{b \argt}^2}(x)$, which
  follows the same lines.
  Nevertheless we will highlight the important changes.

  First, for any point $y \in R_\alpha^1(x)$, we show that we can construct a strategy
  $\bar{\alpha} \in \NASmin$ of player~1 such that
  $\X[y,\bar{\alpha},b]{s} \in R_\alpha^1(x)$ for all controls $b \in \Bcal$ of player~2
  and all times $s \geq 0$.
  Indeed, there exist $\bar{b} \in \Bcal$ and $t \geq 0$ such that
  $y = \X[x,\alpha,\bar{b}]{t}$.
  Then, for any control $b \in \Bcal$, let us introduce the control $\bar{b}|b$
  obtained by concatenating $\bar{b}$ and $b$ in the following way:
  \[
    (\bar{b}|b)_s =
    \begin{cases}
      \bar{b}_s , & \text{if} \enspace s \leq t , \\
      b_{s-t}  , & \text{if} \enspace s > t .
    \end{cases}
  \]
  We further define the strategy $\bar{\alpha}$ of player~1 as follows:
  $\bar{\alpha}[b]_s = \alpha[\bar{b}|b]_{s+t}$ for all $s \geq 0$.
  It is straightforward to verify that $\bar{\alpha}$ is nonanticipating and, moreover,
  that for all $s \geq 0$, $\X[y,\bar{\alpha},b]{s} = \X[x,\alpha,\bar{b}|b]{s+t}$.
  Thus, for all $s \geq 0$ we have $\X[y,\bar{\alpha},b]{s} \in R_\alpha^1(x)$.

  Consider now $z \in \clo{R_\alpha^1}(x) \setminus R_\alpha^1(x)$ and fix some
  $\varepsilon > 0$ and $T \geq 0$.
  There exists $y \in R_\alpha^1(x)$ such that $\Norm{y-z} \leq \varepsilon e^{-L_f T}$.
  Let $\bar{\alpha} \in \NASmin$ be the strategy of player~1 defined above,
  which ensures that $\X[y,\bar{\alpha},b]{s} \in R_\alpha^1(x)$
  for all $b \in \Bcal$ and $s \geq 0$.
  We have the following standard estimate on the trajectories of~\labelcref{eq:controlled-system}:
  \begin{equation*}
    \label{eq:trajectories-estimate-ii}
    \Norm{\X[x,a,b]{t} - \X[y,a,b]{t}} \leq e^{L_f t} \Norm{x-y} , \quad
    \forall x, y \in \R^n , \quad \forall a \in \Acal , \quad b \in \Bcal, \quad \forall t \geq 0 ,
  \end{equation*}
  from which we deduce that, for all $s \geq 0$,
  \[
    \Norm{\X[y,\bar{\alpha},b]{s} - \X[z,\bar{\alpha},b]{s}} \leq
    e^{L_f s} \Norm{y-z} \leq \varepsilon e^{L_f (s-T)} \leq \varepsilon .
  \]
  Thus, for all $b \in \Bcal$ and $s \in [0,T]$ we have
  $\dist{R_\alpha^1(x)}{\X[z,\bar{\alpha},b]{s}} \leq \varepsilon$,
  which finally proves that $\clo{R_\alpha^1}(x)$ is a dominion for player~1.

  For $R_{b \argt}^2(x)$ the proof is identical, up to the changes in players' role.
  The main difference concerns the construction, for any point $y \in R_{b \argt}^2(x)$
  and any strategy $\alpha \in \NASmin$ of player~1, of a control $\bar{b} \in \Bcal$ of
  player~2 such that $\X[y,\alpha,\bar{b}]{s} \in R_{b \argt}^2(x)$ for all $s \geq 0$.
  We next detail this construction.
  Let $\bar{\alpha} \in \NASmin$ and $t \geq 0$ be such that
  $y = \X[x,\bar{\alpha},b_{\bar{\alpha}}]{t}$.
  Let us also define, for any $b \in \Bcal$, the control $\sigma_t b$
  by $(\sigma_t b)_s = b_{s+t}$.
  We then define a nonanticipating strategy $\bar{\alpha}|\alpha$ as follows:
  \[
    (\bar{\alpha}|\alpha)[b]_s =
    \begin{cases}
      \bar{\alpha}[b]_s , & \text{if} \enspace s \leq t , \\
      \alpha[\sigma_t b]_{s-t} , & \text{if} \enspace s > t .
    \end{cases}
  \]
  If we set $\bar{b} = \sigma_t b_{\bar{\alpha}|\alpha}$, one can check that
  $\X[y,\alpha,\bar{b}]{s} = \X[x,\bar{\alpha}|\alpha,b_{\bar{\alpha}|\alpha}]{s+t}$
  for all $s \geq 0$ (in particular we have
  $\X[x,\bar{\alpha}|\alpha,b_{\bar{\alpha}|\alpha}]{t} =
  \X[x,\bar{\alpha},b_{\bar{\alpha}}]{t} = y$
  because the map $\alpha \mapsto b_{\alpha}$ is nonanticipating and so
  $(b_{\bar{\alpha}|\alpha})_s = (b_{\bar{\alpha}})_s$ for almost all $s \in [0,t]$).
  Hence the result.
\end{proof}

\section{Operator-theoretic characterization of dominions}
\label{sec:dominion-characterization}

In this final section, we characterize dominions in operator-theoretic terms.
Thus, we show that the notion of dominion coincides with the one of leadership domain and
discriminating domain which appears in viability theory\footnote{We mention that
the notion of discriminating / leadership domain, hence of dominion, relates with the ones of
\textbf{B}-set and approachability in repeated games with vector payoffs.
Indeed, In \cite{ASQS09}, As Soulaimani, Quincampoix and Sorin proved that
the \textbf{B}-sets for one player (which provide a sufficient condition for approachability) 
coincide with the discriminating domains for that player in an associated differential game.}
(see, e.g., \cite{Car96}).
This characterization stems from the similarities that exist between dominions
on the one hand, and the interpretation of discriminating and leadership domains,
on the other hand.
Indeed, the latter, which are originally defined by means of inequalities involving $H_\infty$,
can also be characterized in terms of invariant dynamics (see, e.g., \cite{Car96}).
This correspondence between the two notions can be readily established for leadership
domains and dominions of player~2 in the lower game (see Theorem~2.3, {\em ibid.}).
As for the correspondence between discriminating domains and dominions
of player~1, it is not as straightforward since the interpretation theorem
(Theorem~2.1, {\em ibid.}) requires convexity properties.
Such assumptions -- typically, $A$ must be convex and $f$, affine in $a$ -- are commonly assumed
in viability theory but are not needed here.
Nevertheless, by adapting the proof of the latter result to our setting, we are able to
show that dominions of the first player in $\Gamma^-$ can indeed be characterized
as discriminating domains.
We next state precisely these results.

To this end, we need to introduce the following definition.
A vector $p \in \R^n$ is a {\em proximal normal} to a subset $K$ of $\R^n$
at point $x \in K$ if $\dist{K}{x+p} = \Norm{p}$.
We denote by $\NP{K}{x}$ the set of proximal normals to $K$ at $x$.
Note that, if we let $\Proj{K}{z}$ be the set of projections of any point $z \in \R^n$
onto $K$, i.e.,
\[
  \Proj{K}{z} := \big\{ x \in K \mid \dist{K}{z} = \Norm{x-z} \big\} ,
\]
then the definition of a proximal normal implies that for every vector $p \in \NP{K}{x}$
and every scalar $\nu \in (0,1)$, we have $\Proj{K}{(x + \nu p)} = \{ x \}$.

\medskip
We now provide the operator-theoretic characterizations of dominions.
The first one, for dominions of player~2 in $\Gamma^-$, comes readily from
the correspondence of the latter with leadership domains in viability theory.

\begin{theorem}[{\cite[Thm.~2.3]{Car96}}]
  \label{thm:leadership-domain}
  A nonempty closed set $D$ is a dominion of player~2 in the lower game $\Gamma^-$
  if and only if
  \begin{align*}
    \forall x \in D , \quad \forall p \in \NP{D}{x} , \quad
    & H_\infty(x, -p) \leq 0 , \\
    \text{i.e.,} \quad
    & \min_{b \in B} \max_{a \in A} \scalar{f(x,a,b)}{p} \leq 0 .
  \end{align*}
\end{theorem}

We next give a similar characterization for dominions of player~1,
which relates them with discriminating domains.

\begin{theorem}
  \label{thm:discriminating-domain}
  A nonempty closed set $D$ is a dominion of player~1 in the lower game $\Gamma^-$
  if and only if
  \begin{align*}
    \forall x \in D , \quad \forall p \in \NP{D}{x} , \quad
    & H_\infty(x, p) \geq 0 , \\
    \text{i.e.,} \quad
    & \max_{b \in B} \min_{a \in A} \scalar{f(x,a,b)}{p} \leq 0 .
  \end{align*}
\end{theorem}

\begin{proof}
  We first prove the necessary part and suppose that $D$ is a dominion of player~1.
  Toward a contradiction, let us assume that there exists a positive constant $\eta$,
  some $x \in D$ and some $p \in \NP{D}{x}$ such that 
  \[
    \max_{b \in B} \min_{a \in A} \scalar{f(x,a,b)}{p} \geq \eta > 0 .
  \]
  Since the function $b \mapsto \min_{a \in A} \scalar{f(x,a,b)}{p}$ is upper semicontinuous
  and $B$ is compact, there exists an action $\bar{b} \in B$ such that
  \begin{equation}
    \label{eq:discr-i}
    \forall a \in A , \quad \scalar{f(x,a,\bar{b})}{p} \geq \eta .
  \end{equation}
  Let $b \in \Bcal$ be the constant control equal to $\bar{b}$, i.e.,
  $b_t = \bar{b}$ for all $t \geq 0$.

  Since $D$ is a dominion of player~1, given $\varepsilon > 0$ and $T > 0$ there exists
  a strategy $\alpha \in \NASmin$ such that $\dist{D}{\X[x,\alpha,b]{t}} \leq \varepsilon$
  for all $t \in [0,T]$.
  In order to simplify the notation, let $\X{t} = \X[x,\alpha,b]{t}$.
  Then, for all $t \in [0,T]$, choosing any point $\Y{t}$ in $\Proj{D}{\X{t}}$,
  the set of projections of $\X{t}$ on $D$, we have
  \begin{equation}
    \label{eq:discr-ii}
    \Norm{x + p - \X{t}} \geq \Norm{x + p - \Y{t}} -
    \Norm{\X{t} - \Y{t}} \geq \Norm{p} - \varepsilon ,
  \end{equation}
  where we use the fact that $\Y{t} \in D$ and that
  $\Norm{x + p - \Y{t}} \geq \dist{D}{x+p} = \Norm{p}$ since $p \in \NP{D}{x}$.

  On the other hand, 
  for almost all $t \in [0,T]$ we have
  \begin{align*}
    \frac{1}{2} \frac{d}{dt} \Norm{\X{t} - (x + p)}^2 & = \scalar{\dot{\X{t}}}{\X{t} - (x + p)} \\
    & = \scalar{f(\X{t}, \alpha[b]_t, b_t)}{\X{t} - x} - \scalar{f(\X{t}, \alpha[b]_t, b_t)}{p} \\
    & \leq \Norm{\X{t} - x} \Norm{f(\X{t}, \alpha[b]_t, b_t)} \\
    & \phantom{\leq} + \Norm{p} \Norm{f(\X{t}, \alpha[b]_t, b_t) - f(x, \alpha[b]_t, b_t)}
    - \scalar{f(x, \alpha[b]_t, b_t)}{p} \\
    & \leq M_f (M_f + L_f \Norm{p}) t - \eta .
  \end{align*}
  To establish the last inequality, we used the estimate
  \labelcref{eq:trajectories-estimate-i};
  the Lipschitz continuity of $f$ (with Lipschitz constant $L_f$);
  and \labelcref{eq:discr-i}.
  Let $C = M_f (M_f + L_f \Norm{p})$.
  After integrating the latter inequality we get, for all $t \in [0,T]$,
  \[
    \Norm{\X{t} - (x + p)}^2 - \Norm{p}^2 \leq  C t^2 - 2 \eta t ,
  \]
  which, combined with \labelcref{eq:discr-ii}, yields
  \begin{equation}
    \label{eq:discr-iii}
    \varepsilon^2 - 2 \Norm{p} \varepsilon \leq C t^2 - 2 \eta t .
  \end{equation}
  Note that to square \labelcref{eq:discr-ii}, we need to assume that
  $\varepsilon \leq \Norm{p}$, which is possible because $p$ is different
  from $0$ (otherwise \labelcref{eq:discr-i} would not hold).
  In the latter inequality, the positive constants $\Norm{p}$, $C$ and $\eta$ are fixed,
  whereas $\varepsilon$ and $T$ are arbitrary.
  Hence, by choosing $T = \eta / C$ and rewriting \labelcref{eq:discr-iii}
  with $t = T$ we obtain
  \[
    \varepsilon^2 - 2 \Norm{p} \varepsilon \leq - \frac{\eta^2}{C}
  \]
  which is a contradiction if $\varepsilon$ is small enough.
  This concludes the proof of the necessary part.

  We now prove the sufficient part and assume that for all points $x$ in $D$ and
  all proximal normals $p$ in $\NP{D}{x}$, we have
  \begin{equation}
    \label{eq:discr-iv}
    \max_{b \in B} \min_{a \in A} \scalar{f(x,a,b)}{p} \leq 0 .
  \end{equation}
  We then fix $x \in D$ and positive constants $\varepsilon$ and $T$.
  Our aim is to construct recursively on the subintervals $[t_k,t_{k+1})$ of
  a well-chosen partition $\{ t_k = k \frac{T}{N}\}_{0 \leq k \leq N}$
  of $[0,T]$, a nonanticipating strategy $\alpha$ of player~1 such that
  $\dist{D}{\X[x,\alpha,b]{t}} \leq \varepsilon$ for all $t \in [0,T]$
  and all controls $b$ of player~2.
  The mesh $\theta = \frac{T}{N}$ of the partition (which shall depend only on $x$,
  $\varepsilon$, $T$ and the data of the problem) will be chosen a posteriori,
  so we assume for now that it is fixed.
  Also, for any $z \in \R^n$ we shall fix a point in $\Proj{D}{z}$
  which we denote by $\proj{D}{z}$.

  We start by selecting an arbitrary element $\bar{a}$ in $A$ and set
  $\alpha[b]_t = \bar{a}$ for all $b \in \Bcal$ and $t \in [0,t_1)$.
  Note that $\alpha$ is obviously nonanticipating on $[0,t_1)$, that is,
  for any controls $b^1, b^2 \in \Bcal$ that coincide almost everywhere on $[0,t_1)$,
  we have $\alpha[b^1]_t = \alpha[b^2]_t$ for (almost) all $t \in [0,t_1)$.

  Next we assume that $\alpha$ has been defined on $[0,t_k)$ with $0 < k < N$
  and that it is nonanticipating on this interval.
  Given any control $b \in \Bcal$, if $\X[x,\alpha,b]{t_k} \in D$, then
  we set $\alpha[b]_t = \bar{a}$ on $[t_k,t_{k+1})$.
  Otherwise, letting $\X{k} = \X[x,\alpha,b]{t_k}$ (for simplicity) and
  $\Y{k} = \proj{D}{\X{k}}$, we introduce the set-valued map $\Phi$
  defined from $B$ to $A$ by
  \[
    \forall \varb \in B , \quad
    \Phi(\varb) = \big\{ a \in A \mid \scalar{f(\Y{k},a,\varb)}{\X{k} - \Y{k}} \leq 0 \big\} .
  \]
  Let us observe that $\Phi$ depends on the control $b$ only through $\X{k}$.
  Thus, if two controls $b^1$ and $b^2$ are equal almost everywhere on $[0,t_k)$,
  then $\X[x,\alpha,b^1]{t_k} = \X[x,\alpha,b^2]{t_k}$ and therefore they define
  the same set-valued map.

  Since $f$ is continuous, $\Phi$ is measurable and has closed values.
  Moreover, since $\X{k} - \Y{k} \in \NP{D}{\Y{k}}$ by definition,
  \labelcref{eq:discr-iv} implies that the domain of $\Phi$ is $B$, i.e.,
  $\Phi(\varb)$ is nonempty for all $\varb \in B$.
  Hence, according to the Measurable Selection Theorem (see \cite[Thm.~8.1.3]{AF09}),
  $\Phi$ admits a measurable selection $\phi: B \to A$.
  Then we set $\alpha[b]_t = \phi(b_t)$ for all $t \in [t_k,t_{k+1})$.
  It is readily seen that $\alpha$ is nonanticipating on $[0,t_{k+1})$, whence
  on $[0,T)$ after repeating the induction step until $t_{k+1} = T$.
  For $t = T$, we set $\alpha[b]_T = \bar{a}$ for all $b \in \Bcal$.

  To conclude the proof, it remains to show that
  $\dist{D}{\X[x,\alpha,b]{t}} \leq \varepsilon$ on $[0,T]$
  for every control $b$ of player~2.
  So we fix $b \in \Bcal$ and let $\X{t} = \X[x,\alpha,b]{t}$.
  We also let $\X{k} = \X{t_k}$ and $\Y{k} = \proj{D}{\X{k}}$.
  For all $k \in \{0,\dots,N-1\}$ and for almost all $t \in [t_k, t_{k+1}]$ we have
  \begin{align*}
    \frac{1}{2} \frac{d}{dt} \Norm{\X{t} - \Y{k}}^2
    & = \scalar{f(\X{t}, \alpha[b]_t, b_t)}{\X{t} - \Y{k}} \\
    & = \scalar{f(\X{t}, \alpha[b]_t, b_t) - f(\Y{k}, \alpha[b]_t, b_t)}{\X{t} - \Y{k}} \\
    & \quad + \scalar{f(\Y{k}, \alpha[b]_t, b_t)}{\X{t} - \X{k}}
    + \scalar{f(\Y{k}, \alpha[b]_t, b_t)}{\X{k} - \Y{k}} \\
    & \leq L_f \Norm{\X{t} - \Y{k}}^2 + M_f \Norm{\X{t}-\X{k}}
    + \scalar{f(\Y{k}, \alpha[b]_t, b_t)}{\X{k} - \Y{k}} \\
    & \leq L_f \Norm{\X{t} - \Y{k}}^2 + M_f^2 (t-t_k) .
  \end{align*}
  To establish the latter inequality, we used the estimate
  \labelcref{eq:trajectories-estimate-i} and the fact that
  either $\X{k} \notin D$, in which case
  $\scalar{f(\Y{k}, \alpha[b]_t, b_t)}{\X{k} - \Y{k}} \leq 0$ by definition
  of $\alpha$, or $\X{k} \in D$ which implies $\X{k} - \Y{k} = 0$.
  By integration we then obtain
  \[
    \Norm{\X{t} - \Y{k}}^2 \leq \Norm{\X{k} - \Y{k}}^2 + M_f^2 (t-t_k)^2
    + 2 L_f \int_{t_k}^t \Norm{\X{s} - \Y{k}}^2 ds
  \]
  for all $k \in \{0,\dots,N-1\}$ and $t \in [t_k, t_{k+1}]$.
  Gr\"onwall's inequality yields
  \[
    \Norm{\X{t} - \Y{k}}^2 \leq \big( \Norm{\X{k}-\Y{k}}^2 + M_f^2 (t-t_k)^2 \big)
    e^{2 L_f (t-t_k)}
  \]
  and thus
  \[
    \dist{D}{\X{t}}^2 \leq \big( \dist{D}{\X{k}}^2 + M_f^2 \theta^2 \big) e^{2 L_f \theta} .
  \]

  If we apply the latter inequality to $t =t_{k+1}$, we can use it to show by induction that,
  for all $k \in \{1,\dots,N\}$,
  \[
    \dist{D}{\X{k}}^2 \leq M_f^2 \theta^2 e^{2 L_f \theta}
    \big(1 + \dots + e^{2 L_f (k-1) \theta} \big) .
  \]
  Combining now the last two inequalities, we deduce that
  for all $k \in \{0,\dots,N-1\}$ and all $t \in [t_k, t_{k+1}]$,
  \[
    \dist{D}{\X{t}}^2 \leq M_f^2 \theta^2 e^{2 L_f \theta}
    \big(1 + \dots + e^{2 L_f k \theta} \big)
    = M_f^2 \theta^2 e^{2 L_f \theta} \, \frac{e^{2 L_f (k+1) \theta} - 1}{e^{2 L_f \theta} - 1} .
  \]
  Since $k \theta \leq N \theta = T$ and $(e^x - 1)^{-1} \leq x^{-1}$ if $x > 0$,
  we finally get, for all $t \in [0,T]$,
  \[
    \dist{D}{\X{t}}^2 \leq \frac{M_f^2}{2 L_f} \theta e^{2 L_f \theta}
    \big(e^{2 L_f T} - 1\big) .
  \]
  The proof is complete once we have observed that we can choose the mesh of the partition,
  $\theta = \frac{T}{N}$, depending only on $M_f$, $L_f$, $T$ and $\varepsilon$, so that
  the right-hand side in the latter inequality is lower than $\varepsilon^2$.
\end{proof}

\begin{remark}[Dominions in the upper game and with Isaacs' condition]
  \label{rmk:lower-upper-games}
  Similar characterizations for the dominions in the upper game $\Gamma^+$ can be obtained
  after switching the identity of the players (the fact that one player is minimizing
  and the other maximizing does not come into account here).
  Thus, a nonempty closed set $D$ is a dominion of player~1 (resp., player~2) in $\Gamma^+$
  if and only if
  \begin{align*}
    \forall x \in D , \quad \forall p \in \NP{D}{x} , \quad
    & \min_{a \in A} \max_{b \in B} \scalar{f(x,a,b)}{p} \leq 0 \\
    ( \text{resp.,} \quad
    & \max_{a \in A} \min_{b \in B} \scalar{f(x,a,b)}{p} \leq 0 ) .
  \end{align*}
  As a consequence, the classical min-max inequality yields that a dominion of player~1
  in $\Gamma^+$ is also a dominion in $\Gamma^-$, and symmetrically, a dominion of player~2
  in $\Gamma^-$ is also a dominion in $\Gamma^+$ (see \Cref{ex:no-isaacs} for an illustration
  of this situation).
  These observations are consistent with the fact that player~1 (resp., player~2) has
  more information in the lower game (resp., in the upper game),
  hence has an advantage in this game.
  Furthermore, if Isaacs' condition~\labelcref{eq:Isaacs-condition} applies
  to $H_\infty$, then the set of dominions for each player is the same in the lower and
  the upper game.
\end{remark}

\providecommand{\arxiv}[1]{\href{http://www.arXiv.org/abs/#1}{arXiv:#1}}
\bibliographystyle{amsalpha}
\bibliography{references}

\end{document}